\documentclass[preprint,dvipdfmx,leqno]{elsarticle}
\usepackage{
amsmath,amssymb,amsthm,ascmac,
bm,here,xcolor
,txfonts
}
\usepackage{tikz}
\usetikzlibrary{intersections,calc}
\newtheoremstyle{mydefinition}{}{}{}{}
{\bfseries}{}{\newline}{
\thmname{#1}\thmnumber{\ #2}\thmnote{\quad#3}.}
\theoremstyle{mydefinition}
\newtheorem{Thm}{Theorem}
\newtheorem{Prop}[Thm]{Proposition}
\newtheorem{Lem}{Lemma}
\newtheorem*{Lem*}{Lemma}
\newtheorem{Rem}{Remark}
\makeatletter
\newcommand{\opnorm}{\@ifstar\@opnorms\@opnorm}
\newcommand{\@opnorms}[1]{
\left|\mkern-1.5mu\left|\mkern-1.5mu\left|#1
\right|\mkern-1.5mu\right|\mkern-1.5mu\right|}
\newcommand{\@opnorm}[2][]{
\mathopen{#1|\mkern-1.5mu#1|\mkern-1.5mu#1|}#2
\mathclose{#1|\mkern-1.5mu#1|\mkern-1.5mu#1|}}
{\end{smallmatrix}\right)}
{\end{smallmatrix}\right]}
{\end{smallmatrix}\right\}}
\@addtoreset{equation}{section}
\makeatother
\begin{document}
\begin{frontmatter}
\title{Initial-boundary value problems for complex Ginzburg-Landau equations governed by \(p\)-Laplacian in general domains}
\author[add1]{Takanori Kuroda}
\ead{1d\_est\_quod\_est@ruri.waseda.jp}
\author[add2,fn2]{Mitsuharu \^Otani}
\ead{otani@waseda.jp}
\fntext[fn2]{Partly supported by the Grant-in-Aid for Scientific Research, \# 15K13451, the Ministry of Education, Culture, Sports, Science and Technology, Japan.}
\address[add1]{Major in Pure and Applied Physics,\\ Graduate School of Advanced Science and Engineering, \\ Waseda University, 3-4-1 Okubo Shinjuku-ku, Tokyo, 169-8555, JAPAN}
\address[add2]{Department of Applied Physics, School of Science and Engineering, \\ Waseda University, 3-4-1 Okubo Shinjuku-ku, Tokyo, 169-8555, JAPAN}
\begin{abstract}
In this paper, complex Ginzburg-Landau (CGL) equations governed by p-Lap\-la\-cian are studied.
We discuss the global existence of solutions for the initial-boundary value problem of the equation in general domains.
The global solvability of the initial-boundary value problem for the case when \(p = 2\) is already examined by several authors provided that parameters appearing in CGL equations satisfy a suitable condition.
Our approach to CGL equations is based on the theory of parabolic equations with non-monotone perturbations.
By using this method together with some approximate procedure and a diagonal argument, the global solvability is shown without assuming any growth conditions on the nonlinear terms.
\end{abstract}
\end{frontmatter}
\section{Introduction}
\label{sec-1}
In this paper we are concerned with the following complex Ginzburg-Landau equation governed by \(p\)-Laplacian in a general domain \(\Omega \subset \mathbb{R}^N\) with smooth boundary \(\partial \Omega\):\vspace{-1mm}
\begin{equation}
\tag*{(CGL)\(_p\)}
\left\{
\begin{aligned}
&\partial_t u(t, x) \!-\! (\lambda \!+\! i\alpha)\Delta_p u \!-\! (\kappa \!+\! i\beta)|u|^{q-2}u \!-\! \gamma u \!=\! f(t, x)
&&\hspace{-2mm}\mbox{in}\ (t, x) \in [0, T] \times \Omega,\\
&u(t, x) = 0&&\hspace{-2mm}\mbox{on}\ (t, x) \in [0, T] \times \partial\Omega,\\
&u(0, x) = u_0(x)&&\hspace{-2mm}\mbox{in}\ x \in \Omega,
\end{aligned}
\right.
\end{equation}
where \(\lambda, \kappa > 0\), \(\alpha, \beta, \gamma \in \mathbb{R}\) are parameters; 
\(\Delta_p u = \mathop{{\rm div}}(|\nabla u|^{p-2}\nabla u)\) with \(p > \max\left\{1, \frac{2N}{N + 2}\right\}\); 
\(q \geq 2\); 
\(i = \sqrt{-1}\) is the imaginary unit; 
\(u_0: \Omega \rightarrow \mathbb{C}\) denotes an initial value; 
\(f: \Omega \times [0, T] \rightarrow \mathbb{C}\) (\(T > 0\)) is an external force.
Our unknown function \(u:\overline{\Omega} \times [0,\infty) \rightarrow \mathbb{C}\) is a complex valued function.
In extreme cases, (CGL)\(_p\) gives two well-known equations: 
quasi-linear heat equation (when \(\alpha=\beta=0\)) and nonlinear Schr\"odinger-type equation (when \(\lambda=\kappa=0\)).
Thus the general case of (CGL)\(_p\) could be regarded as ``intermediate'' between nonlinear heat equation and nonlinear Schr\"odinger equation.

As a mathematical model for superconductivity, equation (CGL)\(_2\) ((CGL)\(_p\) with \(p=2\)) was introduced by Landau and Ginzburg in 1950 \cite{GL1}.
Subsequently, it was revealed that many nonlinear partial differential equations can be rewritten in the form of (CGL)\(_2\) (\cite{N1}).
Recently Bekki, Harada and Kanai pointed out that some solutions of (CGL)\(_2\) describe nonlinear traveling waves in a human heart (\cite{BHK1}).

Mathematical studies for the solvability of (CGL)\(_p\) are examined extensively for the case \(p = 2\) by several authors.
The first treatment of the case \(p \neq 2\) is done by Okazawa-Yokota \cite{OY1}.
In \cite{OY1}, they assumed a monotonicity condition on parameters \(\lambda\) and \(\alpha\), that is, \(|\alpha|/\lambda < 2\sqrt{p-1}/(p-2)\) and employ the theory of maximal monotone operators in complex Hilbert spaces.
They also assumed boundedness of domains.

On the other hand, it was proposed in Okazawa-Yokota \cite{OY1} that equation (CGL)\(_2\) can be regarded as a parabolic equation.
Based on this line, the global existence of solutions together with some smoothing effect in general domains is examined in \cite{KOS1}.

In this paper, we show the global solvability of (CGL)\(_p\) in general domains without assuming any upper bound condition on \(q\) and without any additional restriction on parameters \(\lambda, \kappa, \alpha, \beta, p\) such as \(|\alpha|/\lambda < 2\sqrt{p-1}/(p-2)\), i.e., we only assume that parameters lie in the so-called CGL-region (see (\ref{CGL_region})).

To deal with the problem in general domains without excessive assumptions, we cannot directly apply major tools for solving evolution equations: the compactness method, the contraction mapping principle and the monotonicity method.
In fact, for the compactness method, the embedding \({\rm W}^{1,p} \subset {\rm L}^2\) is no longer compact in general domains \(\Omega\); 
for the contraction mapping principle, the Sobolev subcritical condition on \(q\) is needed; 
for using the monotonicity of the operator \(-(\lambda + i\alpha)\Delta_p\), one has to impose more restrictive conditions on \(\lambda\) and \(\alpha\) (cf. Okazawa-Yokota \cite{OY1}).

In order to overcome these difficulties, we first introduce suitable approximate problems for (CGL)\(_p\) and solve the problem in bounded domains \(\{\Omega_k\}_{k \in \mathbb{N}}\) which approximate our domain \(\Omega\) for initial values which are compactly supported in \(\Omega_k\) by applying the compactness method.
Letting \(k \to \infty\) we have a limit function \(U_k \to U\), where \(\{U_k\}_{k \in \mathbb{N}}\) are solutions on \(\Omega_k\).
Then we ensure that \(U\) is the desired solution by combining the diagonal argument, local strong convergences and the standard argument of the convex analysis, more precisely, the definition of subdifferential operators.

This  paper consists of seven sections.
In \S2, we fix some notations and prepare some preliminaries.
Two main results are stated in \S3 and key inequalities are prepared in \S4.
In \S5, we introduce approximate problems for (CGL) and show their solvability.
\S6 and \S7 are devoted to proofs of main results.

\section{Notations and Preliminaries}
\label{sec-2}

In this section, we first fix some notations for formulating (CGL)\(_p\) as an evolution equation in a real product function space based on the following identification:
\[
\mathbb{C} \ni u_1 + iu_2 \mapsto (u_1, u_2)^{\rm T} \in \mathbb{R}^2.
\]
Then define the following:
\[
\begin{aligned}
&(U \cdot V)_{\mathbb{R}^2} :=  u_1 v_1 + u_2 v_2, \qquad U=(u_1, u_2)^{\rm T}, \ V=(v_1, v_2)^{\rm T} \in \mathbb{R}^2,\\[1mm]
&\mathbb{L}^2(\Omega) :={\rm L}^2(\Omega) \times {\rm L}^2(\Omega),\quad (U, V)_{\mathbb{L}^2} := (u_1, v_1)_{{\rm L}^2} + (u_2, v_2)_{{\rm L}^2},\\[1mm]
&\quad U=(u_1, u_2)^{\rm T},\quad V=(v_1, v_2)^{\rm T} \in \mathbb{L}^2(\Omega),\\[1mm]
&\mathbb{L}^p(\Omega) := {\rm L}^p(\Omega) \times {\rm L}^p(\Omega),\quad |U|_{\mathbb{L}^p}^p := |u_1|_{{\rm L}^p}^p + |u_2|_{{\rm L}^p}^p\quad\ U \in \mathbb{L}^p(\Omega)\ (1\leq p < \infty).
\end{aligned}
\]
We use the differential symbols to indicate differential operators which act on each component of \({\mathbb{W}^{1, p}_0}(\Omega)\)-elements:
\[
\begin{aligned}
& D_i = \frac{\partial}{\partial x_i}: \mathbb{W}^{1, p}_0(\Omega) := {\rm W}^{1,p}_0(\Omega) \times {\rm W}^{1,p}_0(\Omega) \rightarrow \mathbb{L}^p(\Omega),\\
&D_i U = (D_i u_1, D_i u_2)^{\rm T} \in \mathbb{L}^p(\Omega) \ (i=1, \cdots, N),\\[2mm]
& \nabla = \left(\frac{\partial}{\partial x_1}, \cdots, \frac{\partial}{\partial x_N}\right): \mathbb{W}^{1, p}_0(\Omega) \rightarrow ({\rm L}^p(\Omega))^{2 N},\\
&\nabla U=(\nabla u_1, \nabla u_2)^T \in ({\rm L}^p(\Omega))^{2 N}.
\end{aligned}
\]

We further define, for \(U=(u_1, u_2)^{\rm T},\ V= (v_1, v_2)^{\rm T},\ W = (w_1, w_2)^{\rm T}\),
\[
\begin{aligned}
&U(x) \cdot \nabla V(x) := u_1(x) \nabla v_1(x) + u_2(x) \nabla v_2(x) \in \mathbb{R}^N,\\[2mm]
&( U(x) \cdot \nabla V(x) ) W(x) := ( u_1(x) ~\! w_1(x) \nabla v_1(x), u_2(x) w_2(x) \nabla v_2(x) )^{\rm T} \in \mathbb{R}^{2N},\\[2mm]
&(\nabla U(x) \cdot \nabla V(x)) := \nabla u_1(x) \cdot \nabla v_1(x) + \nabla u_2(x) \cdot \nabla v_2(x) \in \mathbb{R}^1,\\[2mm]
&|\nabla U(x)| := \left(|\nabla u_1(x)|^2_{\mathbb{R}^N} + |\nabla u_2(x)|^2_{\mathbb{R}^N} \right)^{1/2}.
\end{aligned}
\]

As a realization in \(\mathbb{R}^2\) of the imaginary unit \(i\) in \(\mathbb{C}\), we introduce the following matrix \(I\), which is a linear isometry on \(\mathbb{R}^2\):
\[
I = \begin{pmatrix}
0 & -1\\ 1 & 0
\end{pmatrix}.
\]
We abuse \(I\) for the realization of \(I\) in \(\mathbb{L}^p(\Omega)\), i.e., \(I U = ( - u_2, u_1 )^{\rm T}\) for all \(U = (u_1, u_2)^{\rm T} \in \mathbb{L}^p(\Omega)\).

Then \(I\) satisfies the following properties:

\begin{enumerate}
\item Skew-symmetric property:
\begin{equation}
\label{skew-symmetric_property}
(IU \cdot V)_{\mathbb{R}^2} = - (U \cdot IV)_{\mathbb{R}^2}; \hspace{4mm}
(IU \cdot U)_{\mathbb{R}^2} = 0 \hspace{4mm}
\mbox{for each}\ U, V \in \mathbb{R}^2.
\end{equation}

\item Commutative property with the differential operator \(D_i = \frac{\partial}{\partial x_i}\):
\begin{equation}
\label{commutative_property}
I D_i = D_i I:\mathbb{W}^{1, p}_0 \rightarrow \mathbb{L}^p\ (i=1, \cdots, N).
\end{equation}

\item (In)equalities resulting from the orthogonality of vectors \(V\) and \(IV\):
\begin{align}
\label{consequence_from_orthogonality_1}
&(U \cdot V)_{\mathbb{R}^2}^2 + (U \cdot IV)_{\mathbb{R}^2}^2 = |U|_{\mathbb{R}^2}^2 |V|_{\mathbb{R}^2}^2 \hspace{4mm} \mbox{for each}\ U, V \in \mathbb{R}^2,\\[1mm]
\label{consequence_from_orthogonality_2}
&(U,V)_{\mathbb{L}^2}^2 + (U,IV)_{\mathbb{L}^2}^2 \leq |U|_{\mathbb{L}^2}^2 |V|_{\mathbb{L}^2}^2 \hspace{4mm} \mbox{for each}\ U, V \in \mathbb{L}^2(\Omega).
\end{align}
\end{enumerate}
Properties \eqref{skew-symmetric_property} and \eqref{commutative_property} are obvious.
By virtue of the orthogonality of \(V\) and \(IV\), \eqref{consequence_from_orthogonality_1} is nothing but Pythagorean theorem and \eqref{consequence_from_orthogonality_2} comes from Bessel's inequality.

Let \({\rm H}\) be a Hilbert space and denote by \(\Phi({\rm H})\) the set of all lower semi-continuous convex function \(\phi\) from \({\rm H}\) into \((-\infty, +\infty]\) such that the effective domain of \(\phi\) given by \({\rm D}(\phi) := \{u \in {\rm H}; \ \phi(u) < +\infty \}\) is not empty.
Then for \(\phi \in \Phi({\rm H})\), the subdifferential of \(\phi\) at \(u \in {\rm D}(\phi)\) is defined by
\[
\partial \phi(u) := \{w \in {\rm H};  (w, v - u)_{\rm H} \leq \phi(v)-\phi(u) \hspace{2mm} \mbox{for all}\ v \in {\rm H}\}.
\]
Then \(\partial \phi\) becomes a possibly multivalued maximal monotone operator with domain\\
\({\rm D}(\partial \phi) = \{u \in {\rm H} ;  \partial\phi(u) \neq \emptyset\}\).
However for the discussion below, we have only to consider the case where \(\partial \phi\) is single valued.

We introduce the following amalgam space:
\[
\mathbb{X}_p(\Omega) := \left\{u \in \mathbb{L}^2(\Omega);  \nabla u \in \left(\mathbb{L}^p(\Omega)\right)^N\right\}
\]
with norm
\[
|u|_{\mathbb{X}_p} :=
\left\{
\begin{aligned}
&\left[|u|_{\mathbb{L}^2}^p + |\nabla u|_{\mathbb{L}^p}^p\right]^{1/p}&&\mbox{for}\ p \geq 2,\\
&\left[|u|_{\mathbb{L}^2}^{p'} + |\nabla u|_{\mathbb{L}^p}^{p'}\right]^{1/{p'}}&&\mbox{for}\ \max\left\{1,\frac{2N}{N+2}\right\} < p \leq 2
\end{aligned}
\right.
\]
for all \(u \in \mathbb{X}_p(\Omega)\) and
\[
\frac{1}{p} + \frac{1}{p'} = 1
\]
are H\"older conjugate exponents.

We define \(\mathbb{V}_p(\Omega) := \overline{\mathbb{C}_0^\infty(\Omega)}^{|\cdot|_{\mathbb{X}_p}}\) with a norm \(|\cdot|_{\mathbb{V}_p} := |\cdot|_{\mathbb{X}_p}\), which is also an uniformly convex Banach space since it is a closed subspace of \(\mathbb{X}_p(\Omega)\) (see \cite{AF1} and \ref{AmalSP}).

Now we define two functionals \(\varphi, \ \psi:\mathbb{L}^2(\Omega) \rightarrow [0, +\infty]\) by
\begin{align}
\label{varphi}
&\varphi(U) :=
\left\{
\begin{aligned}
&\frac{1}{p} \displaystyle\int_\Omega |\nabla U(x)|^p dx
&&\mbox{if}\ U \in \mathbb{V}_p(\Omega),\\[3mm]
&+ \infty &&\mbox{if}\ U \in \mathbb{L}^2(\Omega)\setminus\mathbb{V}_p(\Omega),
\end{aligned}
\right.
\\[2mm]
\label{psi}
&\psi(U) :=
\left\{
\begin{aligned}
&\frac{1}{q} \displaystyle\int_\Omega |U(x)|_{\mathbb{R}^2}^q dx
&&\mbox{if}\ U \in \mathbb{L}^q(\Omega) \cap \mathbb{L}^2(\Omega),\\[3mm]
&+\infty &&\mbox{if}\ U \in \mathbb{L}^2(\Omega)\setminus\mathbb{L}^q(\Omega).
\end{aligned}
\right.
\end{align}
We note here that if either \(p > 2\) or \(\Omega\) is bounded, then \({\rm D}(\varphi) = \mathbb{V}_p(\Omega)\) coincides with \(\mathbb{W}^{1, p}_0(\Omega)\).
Then it is easy to see that \(\varphi, \psi \in \Phi(\mathbb{L}^2(\Omega))\) and their subdifferentials are given by
\begin{align}
\label{delvaphi}
&\begin{aligned}[t]
&\partial \varphi(U)=-\Delta_p U = - \nabla(|\nabla U|^{p - 2}\nabla U)\\
&\qquad\mbox{with} \ {\rm D}( \partial \varphi) = \{U \in \mathbb{V}_p(\Omega) ;  \Delta_p U \in \mathbb{L}^2(\Omega)\},\\[2mm]
\end{aligned}\\
\label{delpsi}
&\partial \psi(U) = |U|_{\mathbb{R}^2}^{q-2}U\ {\rm with} \ {\rm D}( \partial \psi) = \mathbb{L}^{2(q-1)}(\Omega) \cap \mathbb{L}^2(\Omega).
\end{align}
Furthermore for any $\mu>0$, we can define Yosida approximations
 \(\partial \varphi_\mu,\ \partial \psi_\mu\) of \(\partial \varphi,\ \partial \psi\) by
\begin{align}
\label{Yosida:varphi}
&\partial \varphi_\mu(U) := \frac{1}{\mu}(U - J_\mu^{\partial \varphi}U) 
= \partial \varphi(J_\mu^{\partial \varphi} U), 
\quad J_\mu^{\partial \varphi} : = ( 1 + \mu \partial \varphi)^{-1},
\\[2mm]
\label{Yosida:psi}
&\partial \psi_\mu(U) := \frac{1}{\mu} (U - J_\mu^{\partial \psi} U) 
= \partial \psi( J_\mu^{\partial \psi} U ), 
\quad  J_\mu^{\partial \psi} : = ( 1 + \mu \partial \psi)^{-1}.
\end{align}
Then it is well known that \(\partial \varphi_\mu, \ \partial \psi_\mu\) are Lipschitz continuous on \(\mathbb{L}^2(\Omega)\) (see \cite{S1}).

Here for later use, we prepare some fundamental properties of \(I\) in connection with \(\partial \varphi,\ \partial \psi,\  \partial \varphi_\mu,\ \partial \psi_\mu\).
\begin{Lem}[(c.f. \cite{KOS1} Lemma 2.1)]\label{Lem:2.1}
The following orthogonality properties hold.
\begin{align}
\label{orth:IU}
&(\partial \varphi(U), I U)_{\mathbb{L}^2} = 0\quad 
\forall U \in {\rm D}(\partial \varphi),\quad 
(\partial \psi(U), I U)_{\mathbb{L}^2} = 0\quad \forall U \in {\rm D}(\partial \psi), 
\\[2mm]
\label{orth:mu:IU}
&(\partial \varphi_\mu(U), I U)_{\mathbb{L}^2} = 0,\quad 
(\partial \psi_\mu(U), I U)_{\mathbb{L}^2} = 0 \quad 
\forall U \in \mathbb{L}^2(\Omega), 
\\[2mm]
\label{orth:Ipsi}
&(\partial \psi(U), I \partial \psi_\mu(U))_{\mathbb{L}^2} = 0\quad 
\forall U \in {\rm D}(\partial \psi).
\end{align}
\end{Lem}
These properties can be verified in the same way as in the proof of Lemma 2.1 of \cite{KOS1}.
However the following orthogonality
\[
(\partial\varphi_\mu(U), I\partial\varphi(U))_{\mathbb{L}^2} = 0
\]
does not hold true anymore, since the nonlinear operator \(\partial\varphi\) fails to be self-adjoint.

Moreover we can show that \(\lambda\partial\varphi(U) + \kappa\partial\psi(U)\) is also represented as a single subdifferential operator.
To see this, we use the following criterion for the maximal monotonicity of a sum of two maximal monotone operators.
\begin{Prop}[(Br\'ezis, H. \cite{B1} Theorem 9)]
\label{angle}
Let \(B\) be maximal monotone in \({\rm H}\) and \(\phi \in \Phi({\rm H})\).
Suppose
\begin{equation}
\label{angle1}
\phi((1+\mu B)^{-1}u) \leq \phi(u), \hspace{4mm} \forall \mu>0, \hspace{2mm} \forall u \in {\rm D}(\phi).
\end{equation}
Then \(\partial \phi + B\) is maximal monotone in  \({\rm H}\).
\end{Prop}

\begin{Lem}
\label{fjdkjslslsllll}
Let \(\phi = \varphi\) and \(B = \partial\psi\) given by {\rm (\ref{varphi})} and {\rm (\ref{delpsi})} respectively, then the inequality {\rm (\ref{angle1})} holds.
\end{Lem}

\begin{proof}
We can show \((1 + \mu \partial \psi)^{-1} \mathbb{C}^1_0(\Omega) \subset \mathbb{C}^1_0(\Omega)\) in the same way as in \cite{KOS1}.
Let \(U_n \in \mathbb{C}^1_0(\Omega)\) and \(U_n \rightarrow U\) in \(\mathbb{V}_p\).
Then \(V_n := (1 + \mu \partial \psi)^{-1}U_n \in \mathbb{C}^1_0(\Omega)\) satisfy
\[
\begin{aligned}
|V_n - V|_{\mathbb{L}^2} = |(1 + \mu \partial \psi)^{-1}U_n - (1+\mu \partial \psi)^{-1}U|_{\mathbb{L}^2} \leq |U_n - U|_{\mathbb{L}^2} \rightarrow 0 \hspace{4mm} \mbox{as}\ n \rightarrow \infty,
\end{aligned}
\]
whence it follows that \(V_n \rightarrow V\) in \(\mathbb{L}^2(\Omega)\).
Next differentiating \((1 + \mu \partial \psi)V_n := U_n\), we obtain
\begin{equation}
\label{adfjakfjdkaj}
\nabla U_n(x)=
\begin{aligned}[t]
&(1+\mu|V_n(x)|_{\mathbb{R}^2}^{q-2})\nabla V_n(x)\\
& + \mu(q-2)|V_n(x)|_{\mathbb{R}^2}^{q-4} (V_n(x) \cdot \nabla V_n(x)) V_n(x).
\end{aligned}
\end{equation}
Multiplying (\ref{adfjakfjdkaj}) by \(|\nabla V_n(x)|^{p - 2}\nabla V_n\), we get
\[
|\nabla V_n(x)|^p \leq (\nabla U_n(x) \cdot |\nabla V_n(x)|^{p - 2}\nabla V_n(x))\quad\mbox{a.e.}\ x \in \Omega.
\]
We integrate both sides over \(\Omega\) and apply Young's inequality to obtain
\begin{equation}
\label{asasa}
\varphi(V_n) \leq \varphi(U_n).
\end{equation}
Passing to the limit, the following inequality holds by the lower semi-continuity of the norm \(|\cdot|_{\mathbb{L}^p}\):
\[
\varphi(V) \leq \liminf_{n \to \infty}\varphi(V_n) \leq \lim_{n \to \infty}\varphi(U_n) = \varphi(U),
\]
whence follows \(V \in \mathbb{V}_p\) and (\ref{angle1}).
\end{proof}

Now we see that \(\lambda \partial \varphi + \kappa \partial \psi\) is maximal monotone for all \(\lambda, \kappa >0\).
Therefore, since the trivial inclusion \(\lambda \partial \varphi + \kappa \partial \psi \subset \partial (\lambda \varphi + \kappa \psi)\) holds, we obtain the following relation:
\begin{equation}
\label{adfjdhss}
\lambda \partial \varphi + \kappa \partial \psi = \partial (\lambda \varphi + \kappa \psi)\quad\mbox{for all}\ \lambda, \kappa >0.
\end{equation}

Thus (CGL)\(_p\) can be reduced to the following evolution equation:
\[
\tag*{(ACGL)\(_p\)}
\left\{
\begin{aligned}
&\frac{dU}{dt}(t) \!+\! \partial (\lambda \varphi \!+\! \kappa \psi)(U) \!+\! \alpha I \partial \varphi(U) \!+\! \beta I \partial \psi(U) \!-\! \gamma U \!=\! F(t),\quad t \in (0,T),\\
&U(0) =U_0,
\end{aligned}
\right.
\]
where \(f(t, x) = f_1(t, x) + i f_2(t, x)\) is identified with \(F(t) = (f_1(t, \cdot), f_2(t, \cdot))^{\rm T} \in \mathbb{L}^2(\Omega)\).

\section{Main Results}
\label{sec-3}

In order to state our main results, we introduce the CGL-region (cf. \cite{KOS1}) given by:
\begin{equation}
\label{CGL_region}
\begin{aligned}[t]
{\rm CGL}(r)&:= 
\left\{(x, y) \in \mathbb{R}^2 ; xy \geq 0\ \mbox{or}\ \frac{|xy|-1}{|x|+|y|} <r
\right\}\\
&= {\rm S}_1(r) \cup {\rm S}_2(r) \cup\ {\rm S}_3(r) \cup {\rm S}_4(r),
\end{aligned}
\end{equation}
where \(S_i(r)\) (\(i = 1, 2, 3, 4\)) is given by
\begin{align*}
{\rm S}_1(r)&:=\left\{(x,y) \in \mathbb{R}^2; |x| \leq r \right\},
&{\rm S}_2(r)&:=\left\{(x,y)  \in \mathbb{R}^2; |y| \leq r \right\},\\
{\rm S}_3(r)&:=\left\{(x,y) \in \mathbb{R}^2; xy>0 \right\},
&{\rm S}_4(r)&:=\left\{(x,y)  \in \mathbb{R}^2; |1+xy| < r |x-y|\right\}.
\end{align*}
\begin{figure}[H]
\begin{center}
\begin{tikzpicture}[domain=-5:5,scale=0.7]
\filldraw[blue, opacity =.3] (-0.8,-5) -- (0.8,-5) -- (0.8,5) -- (-0.8,5) -- cycle;
\filldraw[magenta, opacity =.3] (-5,-0.8) -- (-5,0.8) -- (5,0.8) -- (5,-0.8) -- cycle;
\filldraw[lime, opacity =.3] (0,0) -- (0,5) -- (5,5) -- (5,0) -- cycle;
\filldraw[lime, opacity =.3] (0,0) -- (0,-5) -- (-5,-5) -- (-5,0) -- cycle;
\filldraw[yellow, opacity =.3, variable=\x] plot[domain=-5:(15/29)](\x, {-0.8-1.64/(\x-0.8)}) -- (-25/21,5) -- plot[domain=(-25/21):-5](\x, {0.8-1.64/(\x+0.8)}) -- cycle;
\filldraw[yellow, opacity =.3, variable=\x] plot[domain=5:(25/21)](\x, {-0.8-1.64/(\x-0.8)}) -- (-15/29,-5) -- plot[domain=(-15/29):5](\x, {0.8-1.64/(\x+0.8)}) -- cycle;
\draw[-latex, thick] (-5,0) -- (5,0);
\draw[-latex, thick] (0,-5) -- (0,5);
\draw[black, dashed] (-5,0.8) -- (5,0.8);
\draw[black, dashed] (-5,-0.8) -- (5,-0.8);
\draw[black, dashed] (0.8,-5) -- (0.8,5);
\draw[black, dashed] (-0.8,-5) -- (-0.8,5);
\draw[black, dashed, variable=\x, domain=-5:(-25/21)] plot(\x, {0.8-1.64/(\x+0.8)});
\draw[black, dashed, variable=\x, domain=(-15/29):5] plot(\x, {0.8-1.64/(\x+0.8)});
\draw[black, dashed, variable=\x, domain=-5:(15/29)] plot(\x, {-0.8-1.64/(\x-0.8)});
\draw[black, dashed, variable=\x, domain=(25/21):5] plot(\x, {-0.8-1.64/(\x-0.8)});
\node[anchor=north east] at (0,0) {{\small\(O\)}};
\node[anchor=north east] at (5,0) {{\small\(x\)}};
\node[anchor=north east] at (0,5) {{\small\(y\)}};
\node[anchor=north west] at (0,0.8) {{\small\(r\)}};
\node[anchor=north east] at (0,-0.8) {{\small\(-r\)}};
\node[anchor=north east] at (0.8,0) {{\small\(r\)}};
\node[anchor=north east] at (-0.8,0) {{\small\(-r\)}};
\node[anchor=south east] at (0.8,0.9) {{\small\({\rm S}_1(r)\)}};
\node[anchor=north west] at (0.8,0.8) {{\small\({\rm S}_2(r)\)}};
\node at (2.0,2.0) {{\small\({\rm S}_3(r)\)}};
\node at (-2.0,-2.0) {{\small\({\rm S}_3(r)\)}};
\node at (1.6,-1.6) {{\small\({\rm S}_4(r)\)}};
\node at (-1.6,1.6) {{\small\({\rm S}_4(r)\)}};
\node[anchor=north west] at (2.0,-2.0) {{\small\(y = -r - \frac{1+r^2}{x-r}\)}};
\node[anchor=south east] at (-2.0,2.0) {{\small\(y = r - \frac{1+r^2}{x+r}\)}};
\node[anchor=north east] at (5,5) {{\small\(\begin{aligned}\left\{\begin{aligned}x &= \frac{\alpha}{\lambda},\\y&=\frac{\beta}{\kappa},\end{aligned}\right.\\r = c_q^{-1}.\end{aligned}\)}};
\end{tikzpicture}
\end{center}
\caption{CGL region}
\end{figure}

Also, we use the parameter \(c_q \in [0, \infty)\) measuring the strength of the nonlinearity:
\begin{equation}
\label{strength_of_nonlinearity}
c_q := \frac{q-2}{2\sqrt{q-1}}.
\end{equation}

We assume that possibly unbounded domain \(\Omega\) have a sequence of bounded subsets with smooth boundary such that
\begin{enumerate}\renewcommand{\labelenumi}{(\roman{enumi})}
\item $\Omega_k \subset \Omega_{k+1} \subset \Omega$ for each $k \in \mathbb{N}$,
\item for all bounded $\Omega' \subset \Omega$ there exists $k \in \mathbb{N}$ such that $\Omega' \subset \Omega_k$,
\end{enumerate}
(see Kuroda-\^Otani-Shimizu \cite{KOS1}).

Then our main results are stated as follows.

\begin{Thm}
\label{main_result_1}
Let \(\Omega \subset \mathbb{R}^N\) be a general domain of uniformly \({\rm C}^2\)-regular class (see, e.g., \cite{A1}) and satisfy (i) and (ii).
Suppose that \(F \in {\rm L}^2(0, T;  \mathbb{L}^2(\Omega))\) with \(T > 0\), \(\max\left\{1, \frac{2N}{2+N}\right\} < p\) and \(\left(\frac{\alpha}{\lambda}, \frac{\beta}{\kappa}\right) \in {\rm CGL}(c_q^{-1})\).
Then for any \(U_0 \in \mathbb{V}_p(\Omega) \cap \mathbb{L}^q(\Omega)\), there exists a solution \(U \in {\rm C}([0, T];  \mathbb{L}^2(\Omega))\) of {\rm (ACGL)\(_p\)} satisfying
\begin{enumerate}
\item \(U \in {\rm W}^{1,2}(0,T;  \mathbb{L}^2(\Omega))\cap{\rm C}([0,T];\mathbb{V}_p(\Omega)\cap\mathbb{L}^q(\Omega))\),
\item \(U(t) \in {\rm D}(\partial \varphi) \cap {\rm D}(\partial \psi)\) for a.e. \(t \in (0, T)\) and satisfies {\rm (ACGL)\(_p\)} for a.e. \(t \in (0, T)\),
\item \(\partial \varphi(U(\cdot)), \partial \psi(U(\cdot)) \in {\rm L}^2(0, T; \mathbb{L}^2(\Omega))\).
\end{enumerate}
\end{Thm}

As for the smoothing effect, the following result holds.
\begin{Thm}
\label{main_result_2}
Let all assumptions in Theorem \ref{main_result_1} be satisfied.
Then for any \(U_0 \in \mathbb{L}^2(\Omega)\), there exists a solution \(U \in {\rm C}([0, T];  \mathbb{L}^2(\Omega))\) of {\rm (ACGL)\(_p\)} satisfying
\begin{enumerate}
\item \(U \in {\rm W}^{1, 2}_{\rm loc}((0, T]; \mathbb{L}^2(\Omega))\cap{\rm C}((0,T];\mathbb{V}_p(\Omega)\cap\mathbb{L}^q(\Omega))\),
\item \(U(t) \in {\rm D}(\partial \varphi) \cap {\rm D}(\partial \psi)\) for a.e. \(t \in (0, T)\) and satisfies {\rm (ACGL)\(_p\)} for a.e. \(t \in (0, T)\),
\item \(\varphi(U(\cdot)), \psi(U(\cdot)) \in {\rm L}^1(0, T)\) and \(t\varphi(U(t)), t\psi(U(t)) \in {\rm L}^\infty(0, T)\),
\item \(\sqrt{t} \frac{d}{dt}U(t), \sqrt{t} \partial \varphi(U(t)), \sqrt{t} \partial \psi(U(t)) \in {\rm L}^2(0, T; \mathbb{L}^2(\Omega))\).
\end{enumerate}
\end{Thm}

To prove Theorems \ref{main_result_1} and \ref{main_result_2}, we need to prepare the following result concerning the bounded domain case:
\begin{Prop}
\label{main_result_bdd}
Let \(\Omega \subset \mathbb{R}^N\) be a bounded domain of \({\rm C}^2\)-regular class.
Suppose that \(F \in {\rm L}^2(0, T;  \mathbb{L}^2(\Omega))\) with \(T > 0\), \(\max\left\{1, \frac{2N}{2+N}\right\} < p\) and \(\left(\frac{\alpha}{\lambda}, \frac{\beta}{\kappa}\right) \in {\rm CGL}(c_q^{-1})\).
Then for any \(U_0 \in \mathbb{W}^{1,p}_0(\Omega) \cap \mathbb{L}^q(\Omega)\), there exists a solution \(U \in {\rm C}([0, T];  \mathbb{L}^2(\Omega))\) of {\rm (ACGL)\(_p\)} satisfying
\begin{enumerate}
\item \(U \in {\rm W}^{1,2}(0,T;  \mathbb{L}^2(\Omega))\cap{\rm C}([0,T];\mathbb{W}_0^{1,p}(\Omega)\cap\mathbb{L}^q(\Omega))\),
\item \(U(t) \in {\rm D}(\partial \varphi) \cap {\rm D}(\partial \psi)\) for a.e. \(t \in (0, T)\) and satisfies {\rm (ACGL)\(_p\)} for a.e. \(t \in (0, T)\),
\item \(\partial \varphi(U(\cdot)), \partial \psi(U(\cdot)) \in {\rm L}^2(0, T; \mathbb{L}^2(\Omega))\).
\end{enumerate}
\end{Prop}

\begin{Rem}
The above result concerning the bounded domain case ameliorate the result of Okazawa-Yokota \cite{OY1}, since we are able to exclude the assumption \(|\alpha|/\lambda < (p-2)/2\sqrt{p-1}\).
\end{Rem}

\section{Key Inequalities}
\label{sec-4}

In this section, we prepare some inequalities, which play an important role in establishing a priori estimates.
The same estimates are obtained in \cite{OY1} within the complex valued functions setting; and in \cite{KOS1} under the framework of the product space of real valued functions.
We follow the strategy in \cite{KOS1}.

\begin{Lem}[(cf. \cite{KOS1} Lemma 4.1)]
\label{key_inequality} 
The following inequalities hold for all $U, V \in {\rm D}(\partial \varphi) \cap {\rm D}(\partial \psi)$:
\begin{align}
\label{adfskdfk}
|(|\nabla U|^{p-2}_{\mathbb{R}^{2N}}\nabla V, \nabla I \partial \psi(V))_{\mathbb{L}^2}|
\leq c_q(|\nabla U|^{p-2}_{\mathbb{R}^{2N}}\nabla V, \nabla \partial \psi(V))_{\mathbb{L}^2},\\[1mm]
\label{key_inequality_2} 
|(\partial \varphi(U), I\partial \psi_\mu(U))_{\mathbb{L}^2}|
\begin{aligned}[t]
&\leq c_{q}(\partial \varphi(U), \partial \psi_\mu(U))_{\mathbb{L}^2}\\
&\leq c_{q}(\partial \varphi(U), \partial \psi(U))_{\mathbb{L}^2} \quad  \forall \mu >0,
\end{aligned}
\end{align}
where $\partial \psi_\mu(\cdot)$ is Yosida approximation of $\partial \psi(\cdot)$ given by \eqref{Yosida:psi}.
\end{Lem}

Here we note that taking \(V = U\) in \eqref{adfskdfk}, we get (cf. (4.1) in \cite{KOS1})
\begin{equation}
\label{key_inequality_1}
|(\partial \varphi(U),  I \partial \psi(U))_{\mathbb{L}^2}|
\leq c_{q}(\partial \varphi(U),  \partial \psi(U))_{\mathbb{L}^2}.
\end{equation}

\begin{proof}
By calculating \(\nabla\partial\psi(V)\), we have
\begin{align}\label{akdfjgfdgfk}
\notag&(|\nabla U|^{p-2}_{\mathbb{R}^{2N}}\nabla V, \nabla \partial \psi(V))_{\mathbb{L}^2}\\
&= \int_\Omega |\nabla U|^{p-2}_{\mathbb{R}^{2N}}\left\{(q-2)|V|_{\mathbb{R}^2}^{q-4} |(V \cdot  \nabla V)|^2_{\mathbb{R}^N} + |V|_{\mathbb{R}^2}^{q-2}|\nabla V|_{\mathbb{R}^{2N}}^2\right\} dx.
\end{align}
Making use of \eqref{skew-symmetric_property} and \eqref{commutative_property}, we obtain
\begin{equation}
\begin{aligned}[b]
&(|\nabla U|^{p-2}_{\mathbb{R}^{2N}}\nabla V, \nabla I \partial \psi(V))_{\mathbb{L}^2}\\
&=
\begin{aligned}[t]
&(q-2)\int_\Omega |\nabla U|^{p-2}_{\mathbb{R}^{2N}}|V|_{\mathbb{R}^2}^{q-4} \left( (V \cdot \nabla V), (IV \cdot \nabla V) \right)_{\mathbb{R}^N} dx\\ 
&+ \int_\Omega |\nabla U|^{p-2}_{\mathbb{R}^{2N}}|V|_{\mathbb{R}^2}^{q-2} ( \nabla V \cdot\nabla I V)dx\\
\end{aligned}\\
\label{faldkfk}
&=(q-2)\int_\Omega |\nabla U|^{p-2}_{\mathbb{R}^{2N}}|V|_{\mathbb{R}^2}^{q-4}\left( (V \cdot \nabla V), (IV \cdot \nabla V) \right)_{\mathbb{R}^N} dx.
\end{aligned}
\end{equation}

Here by direct calculations, we note 
\begin{equation}
\label{Id:nablaU}
|(V \cdot \nabla V)|_{\mathbb{R}^N}^2 + |(IV \cdot \nabla V)|_{\mathbb{R}^N}^2
= |V|_{\mathbb{R}^2}^2|\nabla V|_{\mathbb{R}^{2N}}^2.
\end{equation}

Then by Young's inequality, \eqref{faldkfk}, \eqref{Id:nablaU} and \eqref{akdfjgfdgfk}, we obtain  
\[
\begin{aligned}
&|(|\nabla U|^{p-2}_{\mathbb{R}^{2N}}\nabla V, \nabla I \partial \psi(V))_{\mathbb{L}^2}|\\
&\leq (q-2)\int_\Omega |\nabla U|^{p-2}_{\mathbb{R}^{2N}}|V|_{\mathbb{R}^2}^{q-4} |(V \cdot \nabla V)|_{\mathbb{R}^N} \cdot | (IV \cdot \nabla V)|_{\mathbb{R}^N} dx\\
&\leq (q-2) \int_\Omega |\nabla U|^{p-2}_{\mathbb{R}^{2N}}|V|_{\mathbb{R}^2}^{q-4} \frac{1}{2\sqrt{q-1}} \left\{(q-1)|(V \cdot \nabla V)|_{\mathbb{R}^N}^2\!+\!|(IV \cdot \nabla V)|_{\mathbb{R}^N}^2\right\}\! dx\\
&=c_q \int_\Omega |\nabla U|^{p-2}_{\mathbb{R}^{2N}}|V|_{\mathbb{R}^2}^{q-4} \left\{ (q-2)|(V \cdot \nabla V)|_{\mathbb{R}^N}^2 + |V|_{\mathbb{R}^2}^2 |\nabla V|_{\mathbb{R}^{2N}}^2\right\} dx\\
&=c_q(|\nabla U|^{p-2}_{\mathbb{R}^{2N}}\nabla V, \nabla \partial \psi(V))_{\mathbb{L}^2}, 
\end{aligned}
\]
whence follows \eqref{adfskdfk}.

Let $V:=(1+\mu \partial \psi)^{-1}U$, then applying integration by parts, \eqref{skew-symmetric_property} and \eqref{commutative_property}, we have
\begin{align}
\notag
&(\partial \varphi(U), I \partial \psi_\mu(U))_{\mathbb{L}^2}\\[1mm]
\notag
&= (|\nabla U|^{p-2}_{\mathbb{R}^{2N}}\nabla U, \nabla I \partial \psi(V))_{\mathbb{L}^2}\\[1mm]
\notag
&=(|\nabla U|^{p-2}_{\mathbb{R}^{2N}}\nabla V + \mu|\nabla U|^{p-2}_{\mathbb{R}^{2N}}\nabla \partial \psi(V), {\nabla I \partial \psi(V)})_{\mathbb{L}^2}\\[1mm]
\label{aaaiskkxcgqo} 
&=(|\nabla U|^{p-2}_{\mathbb{R}^{2N}}\nabla V, \nabla I \partial \psi(V))_{\mathbb{L}^2}.
\end{align}

Hence by \eqref{aaaiskkxcgqo} and \eqref{adfskdfk}, we obtain
\[
\begin{aligned}
&|(\partial \varphi(U), I \partial \psi_\mu(U))_{\mathbb{L}^2}|\\
&\leq c_q (|\nabla U|^{p-2}_{\mathbb{R}^{2N}}\nabla V, \nabla \partial \psi(V))_{\mathbb{L}^2}\\
&\leq c_q (|\nabla U|^{p-2}_{\mathbb{R}^{2N}}(\nabla V + \mu \nabla \partial \psi(V)), \nabla \partial \psi(V))_{\mathbb{L}^2}
=c_q (\partial \varphi(U), \partial \psi_\mu(U))_{\mathbb{L}^2},
\end{aligned}
\]
which is the first inequality of \eqref{key_inequality_2}.
Finally we show the second inequality of \eqref{key_inequality_2}.
We first note, for a.e. $x \in \Omega$, (see \cite{KOS1})
\begin{align}
\label{VU}
&|V(x)|_{\mathbb{R}^2}
\leq |U(x)|_{\mathbb{R}^2},\\
\label{nablaVU}
&|\nabla V(x)|_{\mathbb{R}^{2N}}
\leq |\nabla U(x)|_{\mathbb{R}^{2N}}, \\
\label{VnablaV}
&\begin{aligned}
|(V(x) \cdot \nabla V(x))|_{\mathbb{R}^N}
&\leq |(V(x) \cdot \nabla U(x))|_{\mathbb{R}^N}\\
&= \frac{|V(x)|_{\mathbb{R}^2}}{|U(x)|_{\mathbb{R}^2}} |(U(x) \cdot \nabla U(x))|_{\mathbb{R}^N} 
\leq |(U(x) \cdot \nabla U(x))|_{\mathbb{R}^N}.
\end{aligned}
\end{align}

We use \eqref{VU}, \eqref{nablaVU} and \eqref{VnablaV} to get
\[
\begin{aligned}
&(\partial \varphi(U), \partial \psi_\mu(U))_{\mathbb{L}^2}\\
&= \int_\Omega
\begin{aligned}[t]
&\left\{(q-2)|\nabla U|^{p-2}_{\mathbb{R}^{2N}}|V|_{\mathbb{R}^2}^{q-4} \bigl( (V\cdot\nabla V ), (V \cdot \nabla U) \bigr)_{\mathbb{R}^N}
\right.\\
&\ \left. + |\nabla U|^{p-2}_{\mathbb{R}^{2N}}|V|_{\mathbb{R}^2}^{q-2} (\nabla V \cdot \nabla U)\right\} dx
\end{aligned}\\
&\leq\int_\Omega\left\{ (q - 2)|\nabla U|^{p-2}_{\mathbb{R}^{2N}}|U|_{\mathbb{R}^2}^{q-4}| (U \cdot \nabla U )|^2_{\mathbb{R}^N} + |U|_{\mathbb{R}^2}^{q-2}| \nabla U |_{\mathbb{R}^{2N}}^p\right\} dx\\
&= (\partial \varphi(U), \partial \psi(U))_{\mathbb{L}^2}.
\end{aligned}
\]
Therefore we obtain the second inequality of \eqref{key_inequality_2}.
\end{proof}

\section{Bounded Domain Case}
\label{sec-5}
In this section, we prove Proposition \ref{main_result_bdd}, which are concerned with the bounded domain case.
In this case, we can use the compactness argument to deduce a strong convergence of a sequence of solutions of approximate equations.

First we consider the following auxiliary equation:
\[
\tag*{(AE)}
\left\{
\begin{aligned}
&\frac{dU}{dt}(t) + \partial (\lambda \varphi + \kappa \psi)(U) + \alpha I \partial \varphi(U) + B(U) = F(t),\ t \in (0, T),\\
&U(0) =U_0,
\end{aligned}
\right.
\]
where \(\beta I \partial \psi(U) - \gamma U\) in (ACGL)\(_p\) is replaced by a Lipschitzian perturbation \(B(U)\) whose Lipschitz constant is denoted by \(L_B\).
As for the global solvability of (AE), the following statements hold:
\begin{Prop}
\label{solvability_of_AE}
Let \(\Omega \subset \mathbb{R}^N\) be a bounded domain of \({\rm C}^2\)-regular class.
Suppose that \(F \in {\rm L}^2(0, T;  \mathbb{L}^2(\Omega))\) with \(T > 0\), \(\max\left\{1, \frac{2N}{2+N}\right\} < p\) and \(\lambda, \kappa > 0\), \(\alpha \in \mathbb{R}\).
Then for any \(U_0 \in \mathbb{W}^{1,p}_0(\Omega) \cap \mathbb{L}^q(\Omega)\), there exists a solution \(U \in {\rm C}([0,T] ;  \mathbb{L}^2(\Omega))\) of {\rm (AE)} satisfying
\begin{enumerate}
\item \(U \in {\rm W}^{1,2}(0,T;  \mathbb{L}^2(\Omega))\),
\item \(U(t) \in {\rm D}(\partial \varphi) \cap {\rm D}(\partial \psi)\) for a.e. \(t \in (0, T)\) and satisfies {\rm (AE)} for a.e. \(t \in (0, T)\),
\item \(\partial \varphi(U(\cdot)), \partial \psi(U(\cdot)) \in {\rm L}^2(0, T; \mathbb{L}^2(\Omega))\).
\end{enumerate}
\end{Prop}

In order to prove Proposition \ref{solvability_of_AE},  we consider the following approximate equation with $ \partial \varphi(U)$ replaced by  its Yosida approximation $\partial \varphi_\nu(U) = \partial \varphi((1+\nu \partial \varphi)^{-1}U)$.
\[
\tag*{(AE)\(_\nu\)}
\left\{
\begin{aligned}
&\frac{dU}{dt}(t) + \partial (\lambda \varphi + \kappa \psi)(U) + \alpha I \partial \varphi_\nu(U) + B(U) = F(t),\ t \in (0, T),\\
&U(0) =U_0,
\end{aligned}
\right.
\]
Since the monotonicity of \(I\partial\varphi\) fails for \(p \neq 2\), we cannot follow the standard theory of monotone perturbations (cf. \cite{KOS1}).

First we prove Proposition \ref{solvability_of_AE} for the case where \(|\alpha| \leq \lambda/2\).
By the standard theory of maximal monotone operators (cf. Br\'ezis \cite{B1}), we have solutions \(U_\nu = U\) of {\rm (AE)\(_\nu\)} satisfying all regularities stated in Proposition \ref{solvability_of_AE}.

Here we establish some a priori estimates.
\begin{Lem}
\label{1st_energy_AE_nu}
Let \(|\alpha| \leq \lambda/2\) and 
$U$ be a solution of {\rm (AE)\(_\nu\)}.
Then there exists a positive constant $C_1$ depending only on $\lambda, \kappa, L_B, T, |B(0)|_{\mathbb{L}^2}$, $|U_0|_{\mathbb{L}^2}$ and $\int_0^T |F|_{\mathbb{L}^2}^2 dt$ satisfying
\begin{equation}\label{1st_energy_AE_nu_1}
\sup_{t \in [0,T]}|U(t)|_{\mathbb{L}^2}^2 \leq C_1.
\end{equation}
\end{Lem}

\begin{proof}
We multiply (AE)\(_\nu\) by its solution \(U\) and integrate with respect to \(t\) on \([0, T]\) to obtain by \eqref{orth:mu:IU}
\begin{align}\label{5.2}
\notag&\frac{1}{2}|U(t)|_{\mathbb{L}^2}^2 + p\lambda\int_0^T\varphi(U(t))dt + q\kappa\int_0^T\psi(U(t))dt\\
&\leq (L_B+1)|U(t)|_{\mathbb{L}^2}^2 +\frac{1}{2}|B(0)|_{\mathbb{L}^2}^2 +\frac{1}{2}|F(t)|_{\mathbb{L}^2}^2,
\end{align}
where we use $|(B(U), U)_{\mathbb{L}^2}| \leq (L_B+\frac{1}{2})|U|_{\mathbb{L}^2}^2 + \frac{1}{2}|B(0)|_{\mathbb{L}^2}^2$ and \eqref{orth:mu:IU}.
By the Gronwall inequality, we conclude \eqref{1st_energy_AE_nu_1}.
\end{proof}

\begin{Lem}
\label{2nd_energy_AE_nu}
Let \(|\alpha| \leq \lambda/2\) and 
$U$ be a solution of ${\rm (AE)}_\nu$.
Then there exists a positive constant $C_2$ depending only on $\lambda, \kappa, L_B, T,\varphi(U_0)$, $\psi(U_0), |B(0)|_{\mathbb{L}^2}$, $|U_0|_{\mathbb{L}^2}$ and $\int_0^T |F|_{\mathbb{L}^2}^2 dt$ satisfying
\begin{align}
\label{2nd_energy_AE_nu_1}
\int_0^T \left| \frac{dU}{dt}(t)\right|_{\mathbb{L}^2}^2 dt + \int_0^T |\partial \varphi(U(t))|_{\mathbb{L}^2}^2 dt + \int_0^T |\partial \psi(U(t))|_{\mathbb{L}^2}^2 dt \leq C_2.
\end{align}
\end{Lem}

\begin{proof}
Multiplying ${\rm (AE)}_\mu$ by $\partial \varphi(U(t))$, we have for a.e. $t \in (0, T)$,
\[
\begin{aligned}
&
\begin{aligned}
&\frac{d}{dt} \varphi(U(t))
+ \lambda |\partial \varphi(U(t))|_{\mathbb{L}^2}^2\\
&\quad + \kappa (\partial \varphi(U(t)), \partial \psi(U(t)))_{\mathbb{L}^2}
+ \alpha (I \partial \varphi_\nu(U(t)), \partial \varphi(U(t)))_{\mathbb{L}^2} 
\end{aligned}\\
&=-(B(U(t)), \partial \varphi(U(t)))_{\mathbb{L}^2}
+ (F(t), \partial \varphi(U(t)))_{\mathbb{L}^2}\\
&\leq \frac{\lambda}{4}|\partial \varphi(U(t))|_{\mathbb{L}^2}^2
+ \frac{2}{\lambda}
\left\{2L_B^2|U(t)|_{\mathbb{L}^2}^2
+ 2|B(0)|_{\mathbb{L}^2}^2
+|F(t)|_{\mathbb{L}^2}^2\right\}
\end{aligned}
\]
Note that \eqref{key_inequality_1} implies $(\partial \varphi(U), \partial \psi(U))_{\mathbb{L}^2} \geq 0$ and
\[
\alpha (I \partial \varphi_\nu(U(t)), \partial \varphi(U(t)))_{\mathbb{L}^2} \geq -|\alpha||\partial\varphi(U)|_{\mathbb{L}^2}^2 \geq -\frac{\lambda}{2}|\partial\varphi(U)|_{\mathbb{L}^2}^2.
\]
Hence by Lemma \ref{1st_energy_AE_nu}, we have
\begin{equation}\label{ajdhcbsdhw}
\frac{d}{dt} \varphi(U(t))
+ \frac{\lambda}{4}|\partial\varphi(U(t))|_{\mathbb{L}^2}^2
\leq \frac{2}{\lambda}
\left\{2L_B^2C_1
+ 2|B(0)|_{\mathbb{L}^2}^2
+ |F(t)|_{\mathbb{L}^2}^2
\right\}.
\end{equation}
Then the integration of \eqref{ajdhcbsdhw} over $(0, t)$ with respect to $t \in (0, T]$ gives
\begin{equation}\label{iiryvxzmkdhqy}
\begin{aligned}
&\varphi(U(t))
+\frac{\lambda}{4} \int_0^t |\partial\varphi(U(s))|_{\mathbb{L}^2}^2ds\\
&\leq \varphi(U_0)
+ \frac{2}{\lambda}
\left\{2L_B^2 C_1T
+ 2T |B(0)|_{\mathbb{L}^2}^2
+ \int_0^T|F|_{\mathbb{L}^2}^2 dt
\right\}.
\end{aligned}
\end{equation}
Next multiplying (AE)\(_\nu\) by $\partial \psi(U(t))$, we have for a.e. $t \in (0, T)$,
\begin{align}
\notag
&\frac{d}{dt} \psi(U(t))
+\lambda (\partial \varphi(U(t)), \partial \psi(U(t)))_{\mathbb{L}^2}
+ \kappa |\partial \psi(U(t))|_{\mathbb{L}^2}^2\\
\notag&= 
\begin{aligned}[t]
&-\alpha (I \partial \varphi_\nu(U(t)), \partial \psi(U(t)))_{\mathbb{L}^2}\\
&-(B(U(t)), \partial \psi(U(t)))_{\mathbb{L}^2}
+(F(t), \partial \psi(U(t)))_{\mathbb{L}^2}
\end{aligned}\\
\label{iikshrtfgtsavfsdf}
&\leq
\begin{aligned}[t]
&\frac{3\kappa}{4}|\partial \psi(U(t))|_{\mathbb{L}^2}^2
+\frac{\alpha^2}{\kappa} |\partial\varphi(U(t))|_{\mathbb{L}^2}^2\\
&+\frac{1}{\kappa}
\left\{2L_B^2 C_1 + 2|B(0)|_{\mathbb{L}^2}^2
+|F(t)|_{\mathbb{L}^2}^2
\right\}.
\end{aligned}
\end{align}
Therefore the integration of \eqref{iikshrtfgtsavfsdf} on \((0, T)\) with respect to \(t\) together with \eqref{key_inequality_1} yields
\begin{align}
\notag
&\psi(U(t))
+\frac{\kappa}{4} \int_0^t |\partial\psi(U(s))|_{\mathbb{L}^2}^2ds\\
\label{uyrgchjsdadhsdgfs}
&\leq 
\begin{aligned}[t]
&\psi(U_0)
+\frac{\alpha^2}{\kappa}\int_0^t |\partial \varphi(U(t))|_{\mathbb{L}^2}^2\\
&+\frac{1}{\kappa}
\left\{2L_B^2 C_1T
+ 2T|B(0)|_{\mathbb{L}^2}^2
+\int_0^T|F|_{\mathbb{L}^2}^2 dt
\right\}.
\end{aligned}
\end{align}
Thus from \eqref{iiryvxzmkdhqy}, \eqref{uyrgchjsdadhsdgfs} and ${\rm (AE)}_\nu$, we derive \eqref{2nd_energy_AE_nu_1}. 
\end{proof}

Now we are in the position of proving Proposition \ref{solvability_of_AE} for \(|\alpha| \leq \lambda/2\).

\begin{proof}
[Proof of Proposition \ref{solvability_of_AE} for \(|\alpha| \leq \lambda/2\)]\quad
Let $U_\nu$ be a solution of ${\rm (AE)}_\nu$.

First we note \(\mathbb{V}_p = \mathbb{W}_0^{1,p}(\Omega) \subset \mathbb{L}^2(\Omega)\), since \(\Omega\) is bounded and \(\max\left\{1, \frac{2N}{2+N}\right\} < p\).
By Lemmas \ref{1st_energy_AE_nu} and \ref{2nd_energy_AE_nu}, we can apply Ascoli's theorem on \(U_\nu\) so that there exists a subsequence \(\{U_{\nu_n}\}\) and \(U \in {\rm L}^2(0, T;  \mathbb{L}^2(\Omega))\) such that
\begin{alignat}{4}
\label{Uconv}
U_{\nu_n} &\to U&&\quad\mbox{strongly in}\ {\rm C}([0, T];  \mathbb{L}^2(\Omega)),\\
\frac{dU_{\nu_n}}{dt} &\rightharpoonup \frac{dU}{dt}&&\quad\mbox{weakly in}\ {\rm L}^2(0, T;  \mathbb{L}^2(\Omega)),\\
\partial\varphi(U_{\nu_n}) &\rightharpoonup \partial\varphi(U)&&\quad\mbox{weakly in}\ {\rm L}^2(0, T;  \mathbb{L}^2(\Omega)),\\
\label{ppsiUconv}
\partial\psi(U_{\nu_n}) &\rightharpoonup \partial\psi(U)&&\quad\mbox{weakly in}\ {\rm L}^2(0, T;  \mathbb{L}^2(\Omega)),\\
J_{\nu_n}^{\partial\psi}(U_{\nu_n}) &\to U&&\quad\mbox{strongly in}\ {\rm L}^2(0, T;  \mathbb{L}^2(\Omega)),\\
\partial\psi_{\nu_n}(U_{\nu_n}) &\rightharpoonup \partial\psi(U)&&\quad\mbox{weakly in}\ {\rm L}^2(0, T;  \mathbb{L}^2(\Omega)).
\end{alignat}
where we used the demiclosedness of \(\frac{d}{dt}, \partial\varphi\) and \(\partial\psi\).
Thus is \(U\) is the desired solution.
\end{proof}

We next proceed by induction.
Assume Proposition \ref{solvability_of_AE} holds with \(\alpha = \frac{n\lambda}{2}\) for some \(n \in \mathbb{Z}\).
Then we have the solution \(U_\nu = U\) of
\[
\tag*{(AE)\(_\nu^n\)}
\left\{
\begin{aligned}
&
\begin{aligned}[t]
\frac{dU}{dt}(t)& + \partial (\lambda \varphi + \kappa \psi)(U) + \frac{n\lambda}{2}I \partial \varphi(U)\\
& + \alpha_0\partial\varphi_\nu(U) + B(U) = F(t),\ t \in (0, T),
\end{aligned}\\
&U(0) =U_0,
\end{aligned}
\right.
\]
with \(|\alpha_0| \leq \lambda/2\).

For the solution of (AE)\(_\nu^n\), we have again the following a priori estimates.
\begin{Lem}
\label{1st_energy_AE_nu2}
Let \(|\alpha_0| \leq \lambda/2\) and 
$U$ be a solution of {\rm (AE)\(_\nu^n\)}.
Then there exists a positive constant $C_1$ depending only on $\lambda, \kappa, L_B, T, |B(0)|_{\mathbb{L}^2}$, $|U_0|_{\mathbb{L}^2}$ and $\int_0^T |F|_{\mathbb{L}^2}^2 dt$ satisfying
\begin{equation}\label{1st_energy_AE_nu2_1}
\sup_{t \in [0,T]}|U(t)|_{\mathbb{L}^2}^2 \leq C_1.
\end{equation}
\end{Lem}

\begin{proof}
Noting \eqref{orth:IU}, we can verify this in much the same way as in the proof of Lemme \ref{1st_energy_AE_nu}.
\end{proof}

\begin{Lem}
\label{2nd_energy_AE_nu2}
Let \(|\alpha_0| \leq \lambda/2\) and 
$U$ be a solution of ${\rm (AE)}_\nu$.
Then there exists a positive constant $C_2$ depending only on $\lambda, \kappa, \alpha, L_B, T,\varphi(U_0)$, $\psi(U_0), |B(0)|_{\mathbb{L}^2}$, $|U_0|_{\mathbb{L}^2}$ and $\int_0^T |F|_{\mathbb{L}^2}^2 dt$ satisfying
\begin{align}
\label{2nd_energy_AE_nu2_1}
\int_0^T \left| \frac{dU}{dt}(t)\right|_{\mathbb{L}^2}^2 dt + \int_0^T |\partial \varphi(U(t))|_{\mathbb{L}^2}^2 dt + \int_0^T |\partial \psi(U(t))|_{\mathbb{L}^2}^2 dt \leq C_2.
\end{align}
\end{Lem}

\begin{proof}
Noting \eqref{skew-symmetric_property}, we can prove this in much the same way as in the proof of Lemme \ref{2nd_energy_AE_nu}.
\end{proof}

\begin{proof}
[Proof of Proposition \ref{solvability_of_AE}]\quad
We prove by mathematical induction.
For every \(\alpha \in \mathbb{R}\), there exist unique \(n \in \mathbb{Z}\) and \(\alpha_0 \in \left(-\frac{\lambda}{2}, \frac{\lambda}{2}\right)\) such that
\[
\alpha = \frac{n\lambda}{2} + \alpha_0.
\]
By Lemma \ref{1st_energy_AE_nu}, Lemma \ref{2nd_energy_AE_nu} and the proof of Proposition \ref{solvability_of_AE} for \(|\alpha| \leq \lambda/2\), we know that (AE)\(_\nu^n\) with \(n=1\) admits a solution.
Then Lemma \ref{1st_energy_AE_nu2}, Lemma \ref{2nd_energy_AE_nu2} and the same arguments as in the proof of Proposition \ref{solvability_of_AE} for \(|\alpha| \leq \lambda/2\) assure the existence of solution of (AE)\(_\nu^n\) with \(n=2\).
Thus to complete the proof, it suffices to repeat this procedure up to \(n=n\) and apply Lemma \ref{1st_energy_AE_nu2}, Lemma \ref{2nd_energy_AE_nu2} and the same arguments as in the proof of Proposition \ref{solvability_of_AE} for \(|\alpha| \leq \lambda/2\).
\end{proof}

Here we proceed to the proof of Proposition \ref{main_result_bdd}.
We consider the following approximate equation.
\[
\tag*{(AE)\(_\mu\)}
\left\{
\begin{aligned}
&
\begin{aligned}
\frac{dU}{dt}(t) &+ \partial (\lambda \varphi + \kappa \psi)(U)\\ 
& +\alpha I \partial \varphi(U) + \beta I\partial\psi_\mu(U) - \gamma U = F(t),\ t \in (0, T),
\end{aligned}\\
&U(0) =U_0.
\end{aligned}
\right.
\]

By Proposition \ref{solvability_of_AE}, there exists a unique solution \(U_\mu = U\) of (AE)\(_\mu\).
We are going to establish the following a priori estimates of \(U_\mu\) independent of \(\mu\).

\begin{Lem}
\label{1st_energy_AE_mu}
Let $U$ be a solution of (AE)\(_\mu\).
Then there exists a positive constant $C_1$ depending only on $\gamma$, $T$, $|U_0|_{\mathbb{L}^2}$ and $\int_0^T|F|_{\mathbb{L}^2}^2 dt$ satisfying
\begin{align}\label{1st_energy_AE_mu_1}
\sup_{t \in [0,T]}|U(t)|_{\mathbb{L}^2}^2
+\int_0^T \varphi(U(s)) ds
+\int_0^T \psi(U(s)) ds
\leq C_1.
\end{align}
\end{Lem}
\begin{proof}
This lemma is proved in the same way as for Lemmas \ref{1st_energy_AE_nu} and \ref{1st_energy_AE_nu2}.
Here we use \eqref{orth:IU} and \eqref{orth:mu:IU}.
\end{proof}

\begin{Lem}[(cf. \cite{KOS1} Lemma 6.2)]
\label{2nd_energy_AE_mu}
Let $U$ be a solution of (AE)\(_\mu\), and let $\left(\frac{\alpha}{\lambda}, \frac{\beta}{\kappa}\right) \in {\rm CGL}(c_q^{-1})$.
Then there exists a positive constant $C_2$ depending only on $\lambda, \kappa, \alpha, \beta, \gamma$, $T,\varphi(U_0), \psi(U_0)$, $|U_0|_{\mathbb{L}^2}$ and $\int_0^T|F|_{\mathbb{L}^2}^2 dt$ satisfying
\begin{equation}\label{2nd_energy_AE_mu_1}
\begin{aligned}
&\sup_{t \in [0,T]}\varphi(U(t))
+ \int_0^T \left| \frac{dU(t)}{dt}\right|_{\mathbb{L}^2}^2 dt\\
&+ \int_0^T |\partial \varphi(U(t))|_{\mathbb{L}^2}^2 dt
+ \int_0^T |\partial \psi(U(t))|_{\mathbb{L}^2}^ 2dt
\leq C_2.
\end{aligned}
\end{equation}
\end{Lem}
\begin{proof}
Let $V(t):=(1+ \mu \partial \psi)^{-1}U(t)$.
Then using the facts that $ U = V + \mu \partial \psi(V)$, $(\partial \psi(V) \cdot V)_{\mathbb{R}^2} = q \psi(V) \geq 0$ and $\psi(V) + \frac{\mu}{2} |\partial \psi(V)|^2 =: \psi_\mu(U) \leq \psi(U)$, we get
\[
\begin{aligned}
&(\partial \psi(U), \partial \psi_\mu(U))_{\mathbb{L}^2}
=\int_\Omega |U|_{\mathbb{R}^2}^{q-2} |V|_{ \mathbb{R}^2}^{q-2} (U \cdot V)_{\mathbb{R}^2}
\geq \int_\Omega |V|_{\mathbb{R}^2}^{2(q-1)}
=|\partial \psi_\mu(U)|_{\mathbb{L}^2}^2,\\
&(U, \partial \psi_\mu(U))_{\mathbb{L}^2}
= q \psi(V)+\mu|\partial \psi(V)|_{\mathbb{L}^2}^2
=q\psi_\mu(U)-(\frac{q}{2}-1)\mu |\partial\psi(V)|_{\mathbb{L}^2}^2
\leq q \psi(U).
\end{aligned}
\]
Hence by virtue of these properties, multiplication of (AE)\(_\mu\) by $\partial \varphi(U(t))$ and $\partial \psi_\mu(U(t))$ together with  \eqref{skew-symmetric_property} give
\begin{align}
\label{afdidfhafdkjftt}
&\frac{d}{dt} \varphi(U(t))
+ \lambda |\partial \varphi(U)|_{\mathbb{L}^2}^2
+ \kappa G(t)
+ \beta B_\mu(t)
= p \gamma_+ \varphi(U(t))
+ (F, \partial \varphi(U))_{\mathbb{L}^2},\\
\label{fsdksdfk}
&\frac{d}{dt} \psi_\mu(U(t))
+ \kappa |\partial \psi_\mu(U)|_{\mathbb{L}^2}^2
+ \lambda G_\mu(t)
- \alpha  B_\mu(t)
\leq q \gamma_+ \psi(U(t))
+ (F, \partial \psi_\mu(U))_{\mathbb{L}^2},
\end{align}
where $\gamma_+:=\max \{\gamma, 0\}$ and
\[
G:=(\partial \varphi(U), \partial \psi(U))_{\mathbb{L}^2},\
G_\mu:=(\partial \varphi(U), \partial \psi_\mu(U))_{\mathbb{L}^2},\
B_\mu:=(\partial \varphi(U), I \partial \psi_\mu(U))_{\mathbb{L}^2}.
\]
We add \eqref{afdidfhafdkjftt}$\times \delta^2$ to \eqref{fsdksdfk} for some $\delta >0$ to get
\begin{align}
\notag
&\frac{d}{dt} \left\{
\delta^2 \varphi(U) + \psi_\mu(U) \right\}
+ \delta^2 \lambda |\partial \varphi(U)|_{\mathbb{L}^2}^2
+ \kappa |\partial \psi_\mu(U)|_{\mathbb{L}^2}^2\\
\notag
&+\delta^2 \kappa G
+ \lambda G_\mu
+(\delta^2 \beta - \alpha )B_\mu\\
\label{sfdsldflsdfsl}
&\leq \gamma_+ \left\{
p \delta^2 \varphi(U) + q \psi(U)
\right\}
+ (F, \delta^2 \partial \varphi(U) + \partial \psi_\mu(U))_{\mathbb{L}^2}.
\end{align}
Let $\epsilon \in (0, \min \{ \lambda, \kappa \})$ be a small parameter.
By the inequality of arithmetic and geometric means, and the fundamental property \eqref{consequence_from_orthogonality_2}, we have
\begin{align}
\notag
&\delta^2 \lambda |\partial \varphi(U)|_{\mathbb{L}^2}^2
+ \kappa |\partial \psi_\mu(U)|_{\mathbb{L}^2}^2\\
\notag
&=\epsilon
\left\{
\delta^2  |\partial \varphi(U)|_{\mathbb{L}^2}^2
+ |\partial \psi_\mu(U)|_{\mathbb{L}^2}^2
\right\}
+(\lambda- \epsilon) \delta^2 |\partial \varphi(U)|_{\mathbb{L}^2}^2
+(\kappa  - \epsilon) |\partial\psi_\mu(U)|_{\mathbb{L}^2}^2\\
\notag
&\geq \epsilon
\left\{
\delta^2  |\partial \varphi(U)|_{\mathbb{L}^2}^2
+ |\partial \psi_\mu(U)|_{\mathbb{L}^2}^2
\right\}
+2\sqrt{
(\lambda-\epsilon)(\kappa-\epsilon)\delta^2 |\partial\varphi(U)|_{\mathbb{L}^2}^2 |\partial\psi_\mu(U)|_{\mathbb{L}^2}^2
}\\
\label{fdfkskgkd}
&\geq
\epsilon
\left\{
\delta^2  |\partial \varphi(U)|_{\mathbb{L}^2}^2
+ |\partial \psi_\mu(U)|_{\mathbb{L}^2}^2
\right\}
+2\sqrt{
(\lambda-\epsilon)(\kappa-\epsilon)\delta^2 (G_\mu^2+B_\mu^2)
}.
\end{align}
We here recall the key inequality \eqref{key_inequality_2}
\begin{align}
\label{fadfks}
G \geq G_\mu \geq c_q^{-1}|B_\mu|.
\end{align}
Hence \eqref{sfdsldflsdfsl}, \eqref{fdfkskgkd} and \eqref{fadfks} yield
\begin{align}
\notag
&\frac{d}{dt} \left\{
\delta^2 \varphi(U) + \psi_\mu(U)
\right\}
+\epsilon \left\{
\delta^2  |\partial \varphi(U)|_{\mathbb{L}^2}^2
+ |\partial \psi_\mu(U)|_{\mathbb{L}^2}^2
\right\}
+ J(\delta, \epsilon)|B_\mu|\\
\label{sfdsldflsksfdhsl}
&\leq \gamma_+ \left\{
p \delta^2 \varphi(U) + q \psi(U)
\right\}
+ (F, \delta^2 \partial \varphi(U) + \partial \psi_\mu(U))_{\mathbb{L}^2},
\end{align}
where
\[
J(\delta, \epsilon)
:=2 \delta
\sqrt{(1+c_q^{-2})(\lambda-\epsilon)(\kappa-\epsilon)}
+ c_q^{-1}(\delta^2 \kappa+\lambda)
-|\delta^2 \beta - \alpha|.
\]

Now we are going to show that  $(\frac{\alpha}{\lambda}, \frac{\beta}{\kappa}) \in {\rm CGL}(c_q^{-1})$ assures $J(\delta, \epsilon) \geq 0$ for some $\delta$ and $\epsilon$.
By the continuity of $J(\delta, \cdot) : \epsilon \mapsto J(\delta, \epsilon)$, it suffices to show $J(\delta, 0) > 0$ for some $\delta$.
When $\alpha \beta >0$, it is enough to take $\delta=\sqrt{\alpha / \beta}$.
When $\alpha \beta \leq 0$, we have $|\delta^2 \beta - \alpha| =\delta^2 |\beta|+ |\alpha|$.
Hence
\[
J(\delta, 0)
= (c_q^{-1}\kappa -|\beta|)\delta^2
+2 \delta \sqrt{(1+c_q^{-2})\lambda \kappa}
+(c_q^{-1}\lambda-|\alpha|).
\]
Therefore if $|\beta|/ \kappa \leq c_q^{-1}$, we get $J(\delta, 0) > 0$ for sufficiently large $\delta >0$.
If $c_q^{-1} < |\beta| / \kappa$, we find that it is enough to see the discriminant is positive, i.e.,
\begin{equation}
D/4:=(1+c_q^{-2})\lambda \kappa
-(c_q^{-1}\kappa -|\beta|)(c_q^{-1}\lambda-|\alpha|)>0.
\end{equation}
Since
\[
D/4>0
\Leftrightarrow
\frac{|\alpha|}{\lambda}\frac{|\beta|}{\kappa}-1
<c_q^{-1}\left(
\frac{|\alpha|}{\lambda}+\frac{|\beta|}{\kappa}
\right),
\]
the condition $(\frac{\alpha}{\lambda}, \frac{\beta}{\kappa}) \in {\rm CGL}(c_q^{-1})$ yields $D>0$, whence $J(\delta, 0)>0$ for\\ $\delta = 2\sqrt{(1+c_q^{-2})\lambda\kappa}/(|\beta|-c_q^{-1}\kappa)>0$.

Now we take $\delta$ and $\epsilon$ such that $J(\delta, \epsilon) \geq 0$.
Integrating \eqref{sfdsldflsksfdhsl} and using Young's inequality and Lemma \ref{1st_energy_AE_mu}, we obtain
\begin{equation}
\label{fdsfdkfj}
\sup_{t \in [0,T]} \varphi(U(t))
+\int_0^T |\partial\varphi(U(s))|_{\mathbb{L}^2}^2 ds
+\int_0^T |\partial\psi_\mu(U(s))|_{\mathbb{L}^2}^2 ds
\leq C_2,
\end{equation}
where $C_2$ depends on the constants stated in Lemma \ref{2nd_energy_AE_mu}.
We multiply (AE)\(_\mu\) by $\partial \psi(U)$ to get by \eqref{orth:Ipsi}
\begin{align}
\notag
&\frac{d}{dt}\psi(U)
+\kappa |\partial \psi(U)|_{\mathbb{L}^2}^2
+\lambda (\partial \varphi(U), \partial \psi(U))_{\mathbb{L}^2}\\
\notag
&=
\begin{aligned}[t]
&- \alpha(I \partial \varphi(U), \partial \psi(U))_{\mathbb{L}^2}
-\beta (I \partial \psi_\mu(U), \partial \psi(U))_{\mathbb{L}^2}\\
&+ q \gamma \psi(U) + (F, \partial \psi(U))_{\mathbb{L}^2}
\end{aligned}\\
\label{adfhgtyyjhds}
&\leq
\frac{\kappa}{4}|\partial \psi(U)|_{\mathbb{L}^2}^2
+ \frac{\alpha^2}{\kappa}|\partial\varphi(U)|_{\mathbb{L}^2}^2
+ q \gamma_+ \psi(U)
+ \frac{\kappa}{4}|\partial \psi(U)|_{\mathbb{L}^2}^2
+ \frac{1}{\kappa}|F|_{\mathbb{L}^2}^2.
\end{align}

Hence \eqref{key_inequality_1}, \eqref{fdsfdkfj} and the integration of \eqref{adfhgtyyjhds} yield
\begin{equation}
\label{fdsfdfkj}
\int_0^T |\partial \psi(U(s))|_{\mathbb{L}^2}^2 ds
\leq C_2.
\end{equation}
Thus (AE)\(_\mu\) together with \eqref{fdsfdkfj} and \eqref{fdsfdfkj} gives the desired estimate \eqref{2nd_energy_AE_mu_1}.
\end{proof}

\begin{proof}
[Proof of Proposition \ref{main_result_bdd}]\quad
By Lemmas \ref{1st_energy_AE_mu} and \ref{2nd_energy_AE_mu}, we can deduce by Ascoli's theorem
\begin{alignat}{4}
U_{\mu_n} &\to U&&\quad\mbox{strongly in}\ {\rm C}([0, T];  \mathbb{L}^2(\Omega)),\\
\frac{dU_{\mu_n}}{dt} &\rightharpoonup \frac{dU}{dt}&&\quad\mbox{weakly in}\ {\rm L}^2(0, T;  \mathbb{L}^2(\Omega)),\\
\partial\varphi(U_{\mu_n}) &\rightharpoonup \partial\varphi(U)&&\quad\mbox{weakly in}\ {\rm L}^2(0, T;  \mathbb{L}^2(\Omega)),\\
\partial\psi(U_{\mu_n}) &\rightharpoonup \partial\psi(U)&&\quad\mbox{weakly in}\ {\rm L}^2(0, T;  \mathbb{L}^2(\Omega)),\\
\partial\psi_{\mu_n}(U_{\mu_n}) = \partial\psi(V_{\mu_n}) &\rightharpoonup g&&\quad\mbox{weakly in}\ {\rm L}^2(0, T;  \mathbb{L}^2(\Omega)),
\end{alignat}
for some \(g \in {\rm L}^2(0, T;  \mathbb{L}^2(\Omega))\).
Here we used the demiclosedness of \(\frac{d}{dt}, \partial\varphi, \partial\psi\).\vspace{1mm}

In order to ensure \(g = \partial\psi(U)\) it suffices to show \(V_{\mu_n} = (1 + \mu_n\partial\psi)^{-1}U_{\mu_n} \to U\) strongly in \({\rm L}^2(0,T; \mathbb{L}^2(\Omega))\) as \(\mu_n \to 0\).
Indeed we have
\[
\begin{aligned}
|V_{\mu_n} - U|_{{\rm L}^2(0, T;  \mathbb{L}^2(\Omega))}
&\leq |V_{\mu_n} - U_{\mu_n}|_{{\rm L}^2(0, T;  \mathbb{L}^2(\Omega))} + |U_{\mu_n} - U|_{{\rm L}^2(0, T;  \mathbb{L}^2(\Omega))}\\
&= \mu_n|\partial\psi(U_{\mu_n})|_{{\rm L}^2(0, T;  \mathbb{L}^2(\Omega))} + |U_{\mu_n} - U|_{{\rm L}^2(0, T;  \mathbb{L}^2(\Omega))} \to 0,
\end{aligned}
\]
as \(\mu_n \to 0\).
\end{proof}

\section{Proof of Theorem \ref{main_result_1}}
In this section, we prove Theorem \ref{main_result_1} with the aid of Proposition \ref{main_result_bdd}.

Let \(\{U^k_0\}_{k \in \mathbb{N}} \subset \mathbb{V}_p\) such that \(U^k_0 \to U_0\) in \(\mathbb{V}_p\) and \(\mathop{{\rm supp}}U^k_0 \subset \Omega_k\), where \(\Omega_k \subset \Omega\) satisfies (i) and (ii).
Let \(U^k = U\) be solutions of (ACGL)\(_p\) with \(\Omega = \Omega_k\) corresponding to initial data \(U^k_0\) given by Proposition \ref{main_result_bdd}.
Here we can assume without loss of generality that for all \(k \in \mathbb{N}\)
\begin{align}
\label{bddUk0}
|U^k_0|_{\mathbb{L}^2(\Omega_k)} &\leq |U_0|_{\mathbb{L}^2(\Omega)} + 1,\\
\label{bddphiUk0}
\varphi(U^k_0) &\leq \varphi(U_0) + 1,\\
\label{bddpsiUk0}
\psi(U^k_0) &\leq \psi(U_0) + 1.
\end{align}
Then repeating much the same arguments as before, we can deduce a priori estimates similar to those in Lemmas \ref{1st_energy_AE_mu} and \ref{2nd_energy_AE_mu}:
\begin{Lem}
\label{1st_energy_ACGL_k}
Let $U^k$ be a solution of (ACGL)\(_p\) with \(\Omega = \Omega_k\) and initial data \(U^k_0\).
Then there exists a positive constant $C_1$ depending only on $\gamma$, $T$, $|U_0|_{\mathbb{L}^2(\Omega)}$ and $\int_0^T|F|_{\mathbb{L}^2(\Omega)}^2 dt$ satisfying
\begin{align}\label{1st_energy_ACGL_k_1}
\sup_{t \in [0,T]}|U^k(t)|_{\mathbb{L}^2(\Omega_k)}^2
+\int_0^T \varphi(U^k(s)) ds
+\int_0^T \psi(U^k(s)) ds
\leq C_1.
\end{align}
\end{Lem}
\begin{proof}
The same argument as in the proof of Lemma \ref{1st_energy_AE_nu} and \eqref{orth:IU} yields \eqref{1st_energy_ACGL_k_1}.
\end{proof}
\begin{Lem}
\label{2nd_energy_ACGL_k}
Let $U^k$ be a solution of (ACGL)\(_p\) with \(\Omega = \Omega_k\) and initial data \(U^k_0\), and let $\left(\frac{\alpha}{\lambda}, \frac{\beta}{\kappa}\right) \in {\rm CGL}(c_q^{-1})$.
Then for a fixed $T>0$, there exists a positive constant $C_2$ depending only on $\lambda, \kappa, \alpha, \beta, \gamma$, $T,\varphi(U_0), \psi(U_0)$, $|U_0|_{\mathbb{L}^2(\Omega)}$ and $\int_0^T|F|_{\mathbb{L}^2(\Omega)}^2 dt$ satisfying
\begin{align}
\notag
&\sup_{t \in [0,T]}\varphi(U^k(t))
+ \sup_{t \in [0,T]}\psi(U^k(t))+ \int_0^T \left| \frac{dU^k(t)}{dt}\right|_{\mathbb{L}^2(\Omega_k)}^2 dt\\
\label{2nd_energy_ACGL_k_1}
&
+ \int_0^T |\partial \varphi(U^k(t))|_{\mathbb{L}^2(\Omega_k)}^2 dt
+ \int_0^T |\partial \psi(U^k(t))|_{\mathbb{L}^2(\Omega_k)}^ 2dt
\leq C_2.
\end{align}
\end{Lem}
\begin{proof}
We can repeat the same arguments as in the proof of Lemma \ref{2nd_energy_AE_mu} with \(I\partial\psi_\mu(U)\) replaced by \(I\partial\psi(U)\).
\end{proof}

In what follows, we denote by \(\tilde{w}\) or \([w]^\sim\) the zero extension of \(w \in \mathbb{L}^2(\Omega_k)\) to \(\mathbb{L}^2(\Omega)\), i.e.,
\[
\tilde{w}(x) = [w]^\sim(x)=
\left\{
\begin{aligned}
&w(x)&&\mbox{if}\ x \in \Omega_k,\\
&0&&\mbox{if}\ x \in \Omega\setminus\Omega_k.
\end{aligned}
\right.
\]
Then we note that
\[
\left[\frac{d}{dt}U^k(t,x)\right]^\sim = \frac{d}{dt}\tilde{U}^k(t,x),\quad\left[\partial\psi(U^k(t,x))\right]^\sim=\partial\psi(\tilde{U}^k(t,x)).
\]
Therefore, by Lemmas \ref{1st_energy_ACGL_k} and \ref{2nd_energy_ACGL_k}, there exists a subsequence \(\{\tilde{U}^{k_n}\}\) of \(\{\tilde{U}^k\}\) satisfying
\begin{alignat}{4}
\label{Uknw}
\tilde{U}^{k_n} &\rightharpoonup U&&\quad\mbox{weakly in}\ {\rm L}^2(0, T;  \mathbb{L}^2(\Omega)),\\
\tag*{(\arabic{section}.\arabic{equation})\('\)}\tilde{U}^{k_n}(T) &\rightharpoonup U(T)&&\quad\mbox{weakly in}\ \mathbb{L}^2(\Omega),\\
\frac{d\tilde{U}^{k_n}}{dt} &\rightharpoonup \frac{dU}{dt}&&\quad\mbox{weakly in}\ {\rm L}^2(0, T;  \mathbb{L}^2(\Omega)),\\
\label{pphiUknw}
\left[\partial\varphi(U^{k_n})\right]^\sim &\rightharpoonup h&&\quad\mbox{weakly in}\ {\rm L}^2(0, T;  \mathbb{L}^2(\Omega)),\\
\partial\psi(\tilde{U}^{k_n}) &\rightharpoonup g&&\quad\mbox{weakly in}\ {\rm L}^2(0, T;  \mathbb{L}^2(\Omega)),
\end{alignat}
for some \(h, g \in {\rm L}^2(0, T; \mathbb{L}^2(\Omega))\).
Hence we get
\begin{equation}\label{6.10}
\frac{d}{dt}U+\lambda h+\kappa g+\alpha I h + \beta I g - \gamma U = F.
\end{equation}

In the sequel, we are going to confirm that \(h = \partial\varphi(U)\) and \(g = \partial\psi(U)\).

In order to show \(g = \partial\psi(U)\), we follow the strategy given in \cite{KOS1}, i.e., we rely on Ascoli's theorem and the diagonal argument.
To do this, we first note that for any \(l \in \mathbb{N}\) \eqref{2nd_energy_ACGL_k_1} assures
\begin{equation}
\label{equicontinuity_E_mu}
|\tilde{U}^{k_n}(t_2)|_{\Omega_l}-\tilde{U}^{k_n}(t_1)|_{\Omega_l}|_{\mathbb{L}^2(\Omega_l)}
\leq
\int_{t_1}^{t_2}
\left|\frac{d\tilde{U}^{k_n}}{ds}\right|_{\mathbb{L}^2(\Omega)}\hspace{-3mm}ds
\leq \sqrt{C_2} \sqrt{t_2-t_1},
\end{equation}
which implies that \(\{\tilde{U}^{k_n}|_{\Omega_l}\}_{k_n\geq l}\) forms an equicontinuous family in \({\rm C}([0,T]; \mathbb{L}^2(\Omega_l))\) for any \(l \in \mathbb{N}\).
Furthermore, since \eqref{1st_energy_ACGL_k_1} and \eqref{2nd_energy_ACGL_k_1} ensures that \(|\nabla\tilde{U}^{k_n}|_{\Omega_l}|_{\mathbb{L}^p(\Omega_l)}\) and \(|\tilde{U}^{k_n}|_{\Omega_l}|_{\mathbb{L}^2(\Omega_l)}\) are bounded and \(\{\tilde{U}^{k_n}(t)|_{\Omega_l}\}_{k_n\geq l}\) forms a precompact set in \(\mathbb{L}^2(\Omega_l)\).
Hence, by Ascoli's theorem, there exists a subsequence \(\{k_n^1\}\) of \(\{k_n\}\) such that
\[
\tilde{U}^{k_n^1}|_{\Omega_1} \rightarrow U^1\quad\mbox{strongly in}\ {\rm C}([0,T];\mathbb{L}^2(\Omega_1))\quad\mbox{as}\ n \to \infty.
\]
Moreover there exists a subsequence \(\{k_n^2\}\) of \(\{k_n^1\}\) such that
\[
\tilde{U}^{k_n^2}|_{\Omega_2} \rightarrow U^1\quad\mbox{strongly in}\ {\rm C}([0,T];\mathbb{L}^2(\Omega_2))\quad\mbox{as}\ n \to \infty.
\]
Successively we can choose sequences \(\{k_n^{l+1}\}\) of \(\{k_n^l\}\) such that
\[
\begin{aligned}
&\{k_n^1\}_{n \in \mathbb{N}} \supset \{k_n^2\}_{n \in \mathbb{N}} \supset \cdots \supset \{k_n^l\}_{n \in \mathbb{N}} \supset \{k_n^{l+1}\}_{n \in \mathbb{N}} \supset \cdots\\
&\tilde{U}^{k_n^l}|_{\Omega_l} \to U^l\quad\mbox{strongly in}\ {\rm C}([0, T]; \mathbb{L}^2(\Omega_l))\quad\mbox{as}\ n \to \infty.
\end{aligned}
\]

Now we take the diagonal sequence \(\{k_n'\}_{n\in\mathbb{N}} := \{k_n^n\}_{n\in\mathbb{N}}\).
Then we get
\begin{equation}
\label{dfhjdfhjdf}
\tilde{U}^{k'_n}|_{\Omega_l} \rightarrow U^l
\quad\mbox{strongly in}\ {\rm C}([0,T]; \mathbb{L}^2(\Omega_l))
\quad\mbox{as}\ n \rightarrow \infty
\quad\forall l \in \mathbb{N}.
\end{equation}
On the other hand, by \eqref{Uknw}, we find that
\begin{equation}
\label{dfhjdfhjdfg}
\tilde{U}^{k'_n}|_{\Omega_l} \rightharpoonup U|_{\Omega_l}
\quad\mbox{weakly in}\ {\rm L}^2(0,T; \mathbb{L}^2(\Omega_l))
\quad\mbox{as}\ n \rightarrow \infty
\quad\forall l \in \mathbb{N}.
\end{equation}
Thus, by \eqref{dfhjdfhjdf} and \eqref{dfhjdfhjdfg}, we find that \(U^l = U|_{\Omega_l}\ \forall l \in \mathbb{N}\) and
\begin{equation}
\label{dfhjdfhjdfh}
\tilde{U}^{k'_n}|_{\Omega_l} \rightarrow U|_{\Omega_l}
\quad\mbox{strongly in}\ {\rm C}([0,T]; \mathbb{L}^2(\Omega_l))
\quad\mbox{as}\ n \rightarrow \infty
\quad\forall l \in \mathbb{N}.
\end{equation}
Here, by virtue of the demiclosedness of the operator \(U \mapsto \partial\psi(U) = |U|^{q-2}U\) in \({\rm L}^2(0,T; \mathbb{L}^2(\Omega_l))\) for any \(l \in \mathbb{N}\), we can conclude
\[
g(t,x)|_{\Omega_l} = \partial\psi(U(t,x)|_{\Omega_l})\quad\forall l \in \mathbb{N},
\]
whence follows
\begin{equation}\label{6.15}
g(t,x) = \partial\psi(U(t,x))\quad \mbox{a.e.}\ (t,x) \in (0,T)\times\Omega.
\end{equation}

Next we ensure that \(h = \partial\varphi(U)\).

Let \(V\) be an arbitrary element of \({\rm C}([0,T];\mathbb{C}_0^1(\Omega))\), then there exists \(n_0\) such that \(\mathrel{\rm supp}V(t)\) is contained in \(\Omega_{n_0}\) for all \(t \in [0,T]\).
Since \(U^{k'_n}\) is a solution of (ACGL)\(_p\) with \(\Omega=\Omega_n\), from the definition of \(\partial\varphi\), we get
\begin{align}
\label{subdiffphi}
(\left[\partial\varphi(U^{k'_n})\right]^\sim + \tilde{U}^{k'_n}, V - \tilde{U}^{k'_n})_{\mathbb{L}^2(\Omega)} \leq \varphi(V) + \frac{1}{2}|V|_{\mathbb{L}^2(\Omega)}^2 - \varphi(\tilde{U}^{k'_n}) - \frac{1}{2}|\tilde{U}^{k'_n}|_{\mathbb{L}^2(\Omega)}^2&\\
\notag\mbox{for all}\ n \geq n_0.&
\end{align}
Multiplying \eqref{subdiffphi} by \(e^{-2t(\gamma + \lambda)}\) and integrating on \((0, T)\), we obtain
\[
\begin{aligned}
&\int_0^Te^{-2t(\gamma + \lambda)}(\left[\partial\varphi(U^{k'_n}(t))\right]^\sim + \tilde{U}^{k'_n}, V)_{\mathbb{L}^2(\Omega)}dt\\
&- \int_0^Te^{-2t(\gamma + \lambda)}(\left[\partial\varphi(U^{k'_n}(t))\right]^\sim + \tilde{U}^{k'_n}, \tilde{U}^{k'_n}(t))_{\mathbb{L}^2(\Omega)}dt\\
&\leq 
\begin{aligned}[t]
&\int_0^Te^{-2t(\gamma + \lambda)}\left\{\varphi(V)
+ \frac{1}{2}|V|_{\mathbb{L}^2(\Omega)}^2\right\}dt\\
& - \int_0^Te^{-2t(\gamma + \lambda)}\left\{\varphi(\tilde{U}^{k'_n}(t))
+\frac{1}{2}|\tilde{U}^{k'_n}|_{\mathbb{L}^2(\Omega)}^2\right\}dt.
\end{aligned}
\end{aligned}
\]

By \eqref{Uknw} and \eqref{pphiUknw}, we have
\begin{align}
\notag
\int_0^Te^{-2t(\gamma + \lambda)}(\left[\partial\varphi(U^{k'_n}(t))\right]^\sim+ \tilde{U}^{k'_n}, V)_{\mathbb{L}^2(\Omega)}dt
&= (\left[\partial\varphi(U^{k'_n})\right]^\sim + \tilde{U}^{k'_n}, e^{-2t(\gamma + \lambda)}V)_{{\rm L}^2(0, T; \mathbb{L}^2(\Omega))}\\
\label{wlim}
&\to (h + U, e^{-2t(\gamma + \lambda)}V)_{{\rm L}^2(0, T; \mathbb{L}^2(\Omega))}\ \mbox{as}\ n \to \infty.
\end{align}
On the other hand, we have
\[
\partial\varphi(U^{k_n'}(t)) 
= \frac{1}{\lambda}\left[-\frac{d}{dt}U^{k_n'}
-\kappa\partial\psi(U^{k_n'})
-\alpha I \partial\varphi(U^{k_n'})
-\beta I\partial\psi(U^{k_n'})
+\gamma U^{k_n'}+F\right].
\]
Hence it holds that
\begin{align}
\notag
&-\int_0^Te^{-2t(\gamma + \lambda)}(\left[\partial\varphi(U^{k'_n}(t))\right]^\sim + \tilde{U}^{k'_n}, \tilde{U}^{k'_n}(t))_{\mathbb{L}^2(\Omega)}dt\\
\notag
&=
\begin{aligned}[t]
&\frac{1}{2\lambda}\int_0^Te^{-2t(\gamma + \lambda)}\frac{d}{dt}|\tilde{U}^{k'_n}|_{\mathbb{L}^2(\Omega)}^2dt 
+ \frac{q\kappa}{\lambda}\int_0^Te^{-2t(\gamma + \lambda)}\psi(\tilde{U}^{k'_n}(t))dt\\
&- \left(\frac{\gamma}{\lambda} + 1\right)\int_0^Te^{-2t(\gamma + \lambda)}|\tilde{U}^{k'_n}(t)|_{\mathbb{L}^2(\Omega)}^2dt
-\frac{1}{\lambda}\int_0^Te^{-2t(\gamma+\lambda)}(F,\tilde{U}^{k_n'})_{\mathbb{L}^2(\Omega)}dt
\end{aligned}\\
\notag
&=
\begin{aligned}[t]
&\frac{1}{2\lambda}\int_0^T\left[
\frac{d}{dt}e^{-2t(\gamma + \lambda)}|\tilde{U}^{k'_n}|_{\mathbb{L}^2(\Omega)}^2\right]dt\\
&+ \frac{q\kappa}{\lambda}\int_0^Te^{-2t(\gamma + \lambda)}\psi(\tilde{U}^{k'_n}(t))dt
-\frac{1}{\lambda}\int_0^Te^{-2t(\gamma+\lambda)}(F,\tilde{U}^{k_n'})_{\mathbb{L}^2(\Omega)}dt
\end{aligned}\\
\label{lhssubdiffphi}
&=
\begin{aligned}[t]
&\frac{1}{2\lambda}e^{-2T(\gamma + \lambda)}|\tilde{U}^{k'_n}(T)|_{\mathbb{L}^2(\Omega)}^2
- \frac{1}{2\lambda}|\tilde{U}^{k'_n}_0|_{\mathbb{L}^2(\Omega)}^2\\
&+ \frac{q\kappa}{\lambda}\int_0^Te^{-2t(\gamma + \lambda)}\psi(\tilde{U}^{k'_n}(t))dt
-\frac{1}{\lambda}\int_0^Te^{-2t(\gamma+\lambda)}(F,\tilde{U}^{k_n'})_{\mathbb{L}^2(\Omega)}dt
\end{aligned}
\end{align}

By the assumption, it holds that
\begin{equation}
\label{sconvU0}
|U^{k'_n}_0|_{\mathbb{L}^2(\Omega)} \to |U_0|_{\mathbb{L}^2(\Omega)}\quad\mbox{as}\ k'_n \to \infty.
\end{equation}
Moreover, \(0 < \min\{1, e^{-2t(\gamma + \lambda)}\} \leq e^{-2t(\gamma + \lambda)} \leq \max\{1, e^{-2t(\gamma + \lambda)}\}\) implies the multiplier \(e^{-2t(\gamma + \lambda)}\) maintains the norm equivalent to that of \({\rm L}^2(0, T; \mathbb{L}^2(\Omega))\).
Therefore by the weak lower semi-continuity of norms, we have
\begin{align}
\notag
&\limsup_{n \to \infty}\left[-\int_0^Te^{-2t(\gamma + \lambda)}\left\{\varphi(\tilde{U}^{k'_n}(t))+\frac{1}{2}|\tilde{U}^{k'_n}|_{\mathbb{L}^2(\Omega)}^2\right\}dt\right]\\
\notag
&= - \liminf_{n \to \infty}\int_0^Te^{-2t(\gamma + \lambda)}\left\{\varphi(\tilde{U}^{k'_n}(t))+\frac{1}{2}|\tilde{U}^{k'_n}|_{\mathbb{L}^2(\Omega)}^2\right\}dt\\
\label{lscphi}
&\leq - \int_0^Te^{-2t(\gamma + \lambda)}\left\{\varphi(U(t))+\frac{1}{2}|U|_{\mathbb{L}^2(\Omega)}^2\right\}dt,
\end{align}
\begin{align}
\liminf_{n \to \infty}\int_0^Te^{-2t(\gamma + \lambda)}\psi(\tilde{U}^{k'_n}(t))dt
\label{lscpsi}
\geq \int_0^Te^{-2t(\gamma + \lambda)}\psi(U(t))dt
\end{align}
and
\begin{equation}\label{6.22}
\liminf_{n \to \infty}|\tilde{U}^{k'_n}(T)|_{\mathbb{L}^2(\Omega)} 
\geq |U(T)|_{\mathbb{L}^2(\Omega)}.
\end{equation}
Thus, in view of above relations \eqref{6.22}-\eqref{subdiffphi},
\begin{equation}\label{lkjhgf}
\begin{aligned}[b]
&\int_0^T(h+U,e^{-2t(\gamma+\lambda)}V)_{\mathbb{L}^2(\Omega)}dt
+\frac{1}{2\lambda}e^{-2T(\gamma+\lambda)}|U(T)|_{\mathbb{L}^2(\Omega)}^2
-\frac{1}{2\lambda}|U_0|_{\mathbb{L}^2(\Omega)}^2\\
&\quad
+\frac{q\kappa}{\lambda}\int_0^Te^{-2t(\gamma+\lambda)}\psi(U(t))dt
-\frac{1}{\lambda}\int_0^Te^{-2t(\gamma+\lambda)}(F,U)_{\mathbb{L}^2(\Omega)}dt\\
&\leq
\begin{aligned}[t]
&\int_0^Te^{-2t(\gamma+\lambda)}\left\{\varphi(V)+\frac{1}{2}|V|_{\mathbb{L}^2(\Omega)}^2\right\}dt\\
&-\int_0^Te^{-2t(\gamma+\lambda)}\left\{\varphi(U)+\frac{1}{2}|U(t)|_{\mathbb{L}^2(\Omega)}^2\right\}dt.
\end{aligned}
\end{aligned}
\end{equation}

Since we already have \(g = \partial\psi(U)\), it holds that
\begin{equation}\label{lkjhgfd}
(Ig, U)_{\mathbb{L}^2(\Omega)} = 0\quad\mbox{for a.e.}\ t \in (0,T).
\end{equation}

Here we claim that it also holds that
\begin{equation}
\label{IfU0}
(Ih, U)_{\mathbb{L}^2(\Omega)} = 0\quad\mbox{for a.e.}\ t \in (0,T).
\end{equation}
To show \eqref{IfU0}, we use an truncation function \(\eta \in {\rm C}_0^1(\mathbb{R}^N)\), \(0 \leq \eta \leq 1\) such that
\[
\eta(x) =
\left\{
\begin{aligned}
&1&&\mbox{if}\ |x| \leq 1/2,\\
&0&&\mbox{if}\ |x| \geq1,
\end{aligned}
\right.
\]
and define \(\eta_R(x) := \eta(x/R)\).
Then \(\mathop{{\rm supp}}\eta_R \subset {\rm B}_R := \{z;  |z|_{\mathbb{R}^N} \leq R\}\) and \(\eta_R\) satisfies
\[
|\nabla\eta_R|_\infty \leq \frac{|\nabla \eta|_\infty}{R},
\]
where \(|w|_\infty = \mathop{{\rm esssup}}_{x \in \mathbb{R}^n}|w(x)|\) for \(w \in {\rm L}^\infty(\mathbb{R}^N)\).

Multiplying \(I\partial\varphi(U^{k'_n})\) by \(\eta_RU^{k'_n}\) and applying integration by parts, we obtain by \eqref{skew-symmetric_property}
\begin{align}
\notag
(I\left[\partial\varphi(U^{k'_n})\right]^\sim, \eta_R\tilde{U}^{k'_n})_{\mathbb{L}^2(\Omega)}
&= 
(-I\nabla(|\nabla \tilde{U}^{k'_n}|^{p-2}\nabla \tilde{U}^{k'_n}), \eta_R\tilde{U}^{k'_n})_{\mathbb{L}^2(\Omega)}\\
\notag
&=
\begin{aligned}[t]
&(I(|\nabla \tilde{U}^{k'_n}|^{p-2}\nabla \tilde{U}^{k'_n}), \eta_R\nabla \tilde{U}^{k'_n})_{(\mathbb{L}^2(\Omega))^{2N}}\\
& + (I(|\nabla \tilde{U}^{k'_n}|^{p-2}\nabla \tilde{U}^{k'_n}), \nabla \eta_R\tilde{U}^{k'_n})_{(\mathbb{L}^2(\Omega))^{2N}}
\end{aligned}\\
\label{trunc}
&= (I(|\nabla \tilde{U}^{k'_n}|^{p-2}\nabla \tilde{U}^{k'_n}), \nabla\eta_R\tilde{U}^{k'_n})_{(\mathbb{L}^2(\Omega))^{2N}}.
\end{align}

We first consider the case \(p > 2\).
Then by H\"older's inequality, we get
\begin{equation}
\label{innerHol}
\left|(\left[I\partial\varphi(U^{k_n'})\right]^\sim,\eta_R\tilde{U}^{k_n'})_{\mathbb{L}^2(\Omega)}\right|
\leq
|\nabla\tilde{U}^{k_n'}|_{\mathbb{L}^p(\Omega)}^{p-1}
|\tilde{U}^{k_n'}|_{\mathbb{L}^p({\rm B}_R)}|\nabla\eta_R|_\infty.
\end{equation}
Hence, if \(p < N\), we can apply Gagliardo-Nirenberg-type interpolation theorem to have
\begin{equation}
\label{GNtype}
|\tilde{U}^{k_n'}|_{\mathbb{L}^p(\Omega)} \leq C|\nabla \tilde{U}^{k_n'}|_{\mathbb{L}^p(\Omega)}^{1-\theta}|\tilde{U}^{k_n'}|_{\mathbb{L}^2(\Omega)}^\theta,
\end{equation}
with \(\frac{1}{p} = \theta\left(\frac{1}{p} - \frac{1}{N}\right) + (1-\theta)\frac{1}{2}\), i.e., \(\theta = \frac{N(p-2)}{Np+2p-2N}\).
We note that \(\theta \in (0,1)\), since \(p > 2N/(N+2)\).

Then by \eqref{innerHol} and \eqref{GNtype}, we obtain
\begin{equation}
\label{innerineq}
\left|(\left[I\partial\varphi(U^{k_n'})\right]^\sim,\eta_R\tilde{U}^{k_n'})_{\mathbb{L}^2(\Omega)}\right|
\leq
C|\nabla\tilde{U}^{k_n'}|_{\mathbb{L}^p(\Omega)}^{p-\theta}
|\tilde{U}^{k_n'}|_{\mathbb{L}^2(\Omega)}|\nabla\eta|_\infty\frac{1}{R}.
\end{equation}

As for the case \(p\geq N\), we need more delicate arguments.
Let \(\Phi_R\) be a mapping from \({\rm B}_R\) onto \({\rm B}_1\) given by \(\Phi_R: x \mapsto y = x/R\) and for any \(U \in \mathbb{L}^r({\rm B}_R)\), we define \(U_R \in \mathbb{L}^r({\rm B}_1)\) by
\[
U_R(y) = U(Ry)\quad\forall y \in {\rm B}_1.
\]
Then we easily have
\begin{align}
\label{poiu}
\begin{aligned}[t]
|U|_{\mathbb{L}^r({\rm B}_R)} &= 
\left(\int_{{\rm B}_R}|U(x)|^rdx\right)^{1/r}\\
&=\left(\int_{{\rm B}_1}|U_R(y)|^rR^Ndy\right)^{1/r}
=R^{\frac{N}{r}}|U_R|_{\mathbb{L}^r({\rm B}_1)},
\end{aligned}\\
\label{poiuy}
\begin{aligned}[t]
|\nabla_xU|_{\mathbb{L}^r({\rm B}_R)}
&=\left(\int_{{\rm B}_R}|\nabla_x U|^rdx\right)^{1/r}\\
&=\left(\int_{{\rm B}_1}\left|\nabla_y U_R(y)\frac{1}{R}\right|^rR^Ndy\right)^{1/r}
=R^{\frac{N-r}{r}}|\nabla_yU_R|_{\mathbb{L}^r({\rm B}_1)}.
\end{aligned}
\end{align}
Let \(p\geq N\), then by Sobolev's embedding theorem, for all \(r \geq p\) there exists \(K_1 = K_1(r)\) such that
\begin{equation}
\label{Sob}
|U|_{\mathbb{L}^r({\rm B}_1)} \leq K_1\left(|\nabla U|_{\mathbb{L}^p({\rm B}_1)} + |U|_{\mathbb{L}^p({\rm B}_1)}\right)\quad\forall U \in \mathbb{W}^{1,p}({\rm B}_1).
\end{equation}
On the other hand, we get
\begin{equation}
\label{lkjhg}
\begin{aligned}
|U|_{\mathbb{L}^p({\rm B}_R)}
&\leq \left(\int_{{\rm B}_R}|U|^{p-1}|U|dx\right)^{1/p}\\
&\leq\left(\int_{{\rm B}_R}|U|^{2(p-1)}dx\right)^{\frac{1}{2p}}\left(\int_{{\rm B}_R}|U|^2dx\right)^{\frac{1}{2p}}\\
&\leq |U|_{\mathbb{L}^{2(p-1)}({\rm B}_R)}^{\frac{p-1}{p}}|U|_{\mathbb{L}^2(\Omega)}^{\frac{1}{p}}.
\end{aligned}
\end{equation}
Then applying \eqref{Sob} with \(r=2(p-1)> p\), \eqref{poiu}, \eqref{poiuy} with \(r=p\), we obtain
\begin{equation}
\label{lkjh}
\begin{aligned}[t]
|U|_{\mathbb{L}^{2(p-1)}({\rm B}_R)}
&= R^{\frac{N}{2(p-1)}}|U_R|_{\mathbb{L}^{2(p-1)}({\rm B}_1)}\\
&\leq R^{\frac{N}{2(p-1)}}K_1\left(|\nabla U_R|_{\mathbb{L}^p({\rm B}_1)} + |U_R|_{\mathbb{L}^p({\rm B}_1)}\right)\\
&= R^{\frac{N}{2(p-1)}}K_1\left(R^{-\frac{N-p}{p}}|\nabla U|_{\mathbb{L}^p({\rm B}_R)} + R^{-\frac{N}{p}}|U|_{\mathbb{L}^p({\rm B}_R)}\right).
\end{aligned}
\end{equation}
Then substituting \eqref{lkjh} in \eqref{lkjhg}, we get
\[
\begin{aligned}
|U|_{\mathbb{L}^p({\rm B}_R)}
&\leq\left[R^{\frac{N}{2(p-1)}}K_1\left(R^{-\frac{N-p}{p}}|\nabla U|_{\mathbb{L}^p({\rm B}_R)} + R^{-\frac{N}{p}}|U|_{\mathbb{L}^p({\rm B}_R)}\right)\right]^{\frac{p-1}{p}}|U|_{\mathbb{L}^2(\Omega)}^{\frac{1}{p}}\\
&=\left(R^{\theta_1} K_1^{\frac{p-1}{p}}|\nabla U|_{\mathbb{L}^p({\rm B}_R)}^{\frac{p-1}{p}}
+K_1^{\frac{p-1}{p}}R^{\theta_2}|U|_{\mathbb{L}^p({\rm B}_R)}^{\frac{p-1}{p}}\right)|U|_{\mathbb{L}^2(\Omega)}^{\frac{1}{p}},
\end{aligned}
\]
where
\[
\theta_1 = \left(\frac{N}{2(p-1)}-\frac{N-p}{p}\right)\frac{p-1}{p}<1
\Leftrightarrow \frac{2N}{N+2}(<2)<p
\]
and
\[
\theta_2=\left(\frac{N}{2(p-1)}-\frac{N}{p}\right)\frac{p-1}{p}<0
\Leftrightarrow
2<p
\]

Since \eqref{1st_energy_ACGL_k_1} implies that \(|\tilde{U}^{k_n'}|_{\mathbb{L}^2({\rm B}_R)}\) is uniformly bounded and \(\theta_2<0\), there exists (a sufficiently large) \(R_0\) such that
\[
K_1^{\frac{p-1}{p}}R^{\theta_2}|\tilde{U}^{k_n'}|_{\mathbb{L}^2(\Omega)}^{\frac{1}{p}}|\tilde{U}^{k_n'}|_{\mathbb{L}^p({\rm B}_R)}^{\frac{p-1}{p}}
\leq\frac{1}{2}|\tilde{U}^{k_n'}|_{\mathbb{L}^p({\rm B}_R)}+1
\quad\forall R \geq R_0.
\]
Hence we obtain
\[
|\tilde{U}^{k_n'}|_{\mathbb{L}^p({\rm B}_R)}
\leq
2R^{\theta_1}K_1^{\frac{p-1}{p}}|\nabla \tilde{U}^{k_n'}|_{\mathbb{L}^p({\rm B}_R)}^{\frac{p-1}{p}}|\tilde{U}^{k_n'}|_{\mathbb{L}^2(\Omega)}^{\frac{1}{p}}+2
\quad\forall R \geq R_0.
\]
Substituting this into \eqref{innerHol}, we finally deduce
\begin{equation}
\label{innerineq2}
\begin{aligned}
&\left|(\left[I\partial\varphi(U^{k_n'})\right]^\sim,\eta_R\tilde{U}^{k_n'})_{\mathbb{L}^2(\Omega)}\right|\\
&\leq
C|\nabla\tilde{U}^{k_n'}|_{\mathbb{L}^p(\Omega)}^{p-1}
\left\{
|\nabla\tilde{U}^{k_n'}|_{\mathbb{L}^p(\Omega)}^{\frac{p-1}{p}}
|\tilde{U}^{k_n'}|_{\mathbb{L}^2(\Omega)}^{\frac{1}{p}}R^{\theta_1}
+1
\right\}
|\nabla\eta|_\infty R^{-1}.
\end{aligned}
\end{equation}

Next for the case \(\max\{1, 2N/(2+N)\} < p \leq 2\), by H\"older's inequality, we have by \eqref{trunc}
\begin{align}
\notag
\left|(\left[I\partial\varphi(U^{k_n'})\right]^\sim,\eta_R\tilde{U}^{k_n'})_{\mathbb{L}^2(\Omega)}\right|
&\leq |\nabla \tilde{U}^{k'_n}|_{\mathbb{L}^p(\Omega)}^{p-1}|\tilde{U}^{k'_n}|_{\mathbb{L}^p({\rm B}_R)}|\nabla \eta_R|_\infty\\
\notag
&\leq |\nabla \tilde{U}^{k'_n}|_{\mathbb{L}^p(\Omega)}^{p-1}|\tilde{U}^{k'_n}|_{\mathbb{L}^2(\Omega)}|{\rm B}_R|^{\frac{2-p}{2p}}|\nabla \eta_R|_\infty\\
\label{trunc2}
&\leq C|\nabla \tilde{U}^{k'_n}|_{\mathbb{L}^p(\Omega)}^{p-1}|\tilde{U}^{k'_n}|_{\mathbb{L}^2(\Omega)}|\nabla \eta|_\infty R^{\frac{N(2-p)}{2p}-1},
\end{align}
where \(C\) denotes the constant independent of \(R\) and \(k'_n\).
We note that
\[
\frac{N(2-p)}{2p}-1 < 0 \Leftrightarrow \frac{2N}{2+N} < p.
\]

Thus by virtue of \eqref{innerineq}, \eqref{innerineq2} and \eqref{trunc2} together with Lemmas \ref{1st_energy_ACGL_k}, \ref{2nd_energy_ACGL_k}, there exist an appropriate constant \(C\) independent of \(R\), \(k'_n\) and \(\rho > 0\) such that
\begin{equation}
\label{innest}
\left|(\left[I\partial\varphi(U^{k_n'})\right]^\sim,\eta_R\tilde{U}^{k_n'})_{\mathbb{L}^2(\Omega)}\right|
\leq CR^{-\rho}
\quad\forall R \geq R_0.
\end{equation}

First we fix \(R > 0\) and take \(k'_n \to \infty\) in \eqref{innest}, then by \eqref{dfhjdfhjdfh} we have
\begin{equation}
\label{innest1}
|(Ih, \eta_RU)_{\mathbb{L}^2(\Omega)}| \leq CR^{-\rho}
\quad\forall R \geq R_0.
\end{equation}

On the other hand by the fact
\[
|(Ih, \eta_RU)_{\mathbb{R}^2}| \to |(Ih, U)_{\mathbb{R}^2}|\quad\mbox{as}\ R \to \infty,\ \mbox{a.e.}\ \Omega\times(0,T),
\]
and
\[
|(Ih, \eta_RU)_{\mathbb{R}^2}| \leq |\eta|_\infty|h|_{\mathbb{R}^2}|U|_{\mathbb{R}^2} \in {\rm L}^1(\Omega\times(0, T)),
\]
we can apply Lebesgue's dominated convergence theorem to obtain
\begin{equation}
\label{inconv}
\int_0^T|(Ih, \eta_RU)_{\mathbb{L}^2(\Omega)}|dt \to \int_0^T|(Ih, U)_{\mathbb{L}^2(\Omega)}|dt\quad\mbox{as}\ R \to \infty.
\end{equation}
Integrating \eqref{innest1} on \((0, T)\) and then passing to the limit \(R \to \infty\) with \eqref{inconv}, we conclude
\[
\int_0^T|(Ih, U)_{\mathbb{L}^2(\Omega)}|dt = 0,
\]
whence follows \eqref{IfU0}.

Hence, by \eqref{lkjhgf}, \eqref{lkjhgfd} and \eqref{IfU0}, we obtain
\begin{align}
\notag
&
\begin{aligned}[t]
&\lambda\int_0^Te^{-2t(\gamma+\lambda)}\left(\varphi(V)+\frac{1}{2}|V|_{\mathbb{L}^2(\Omega)}^2\right)dt\\
&-\lambda\int_0^Te^{-2t(\gamma+\lambda)}\left(\varphi(U)+\frac{1}{2}|U|_{\mathbb{L}^2(\Omega)}^2\right)dt
\end{aligned}\\
\notag
&\geq
\begin{aligned}[t]
&\int_0^Te^{-2t(\gamma+\lambda)}(\lambda h+\lambda U, V)_{\mathbb{L}^2}dt
+\frac{1}{2}\int_0^T\frac{d}{dt}(e^{-2t(\gamma+\lambda)}|U|_{\mathbb{L}^2(\Omega)}^2)dt\\
&+q\kappa\int_0^Te^{-2t(\gamma+\lambda)}\psi(U(t))dt
-\int_0^Te^{-2t(\gamma+\lambda)}(F,U)_{\mathbb{L}^2(\Omega)}dt
\end{aligned}\\
\label{6.39}
&=
\begin{aligned}[t]
&\int_0^Te^{-2t(\gamma+\lambda)}(\lambda h +\lambda U, V)_{\mathbb{L}^2(\Omega)}dt\\
&+\int_0^Te^{-2t(\gamma+\lambda)}
\left(
\frac{dU}{dt} + \alpha Ih+(\kappa + \beta I)g-(\gamma+\lambda) U-F,U
\right)_{\mathbb{L}^2(\Omega)}dt
\end{aligned}\\
\notag
&=\int_0^Te^{-2t(\gamma+\lambda)}(\lambda h +\lambda U, V-U)_{\mathbb{L}^2(\Omega)}dt,
\end{align}
where we used the fact that (see \eqref{6.10})
\[
\frac{dU}{dt}+\alpha Ih + (\kappa + \beta I)g - \gamma U - F = -\lambda h.
\]

Since \({\rm C}([0,T];\mathbb{C}_0^1(\Omega))\) is dense in \(D_\varphi := \left\{V \in {\rm L}^2(0,T;\mathbb{L}^2(\Omega));\int_0^T\varphi(V(t))dt < +\infty\right\}\), \eqref{6.39} holds true also for any \(V \in D_\varphi\).

Let \(t_0 \in (0,T)\) be Lebesgue point of \(h(\cdot)\) and \(V_0\) be an arbitrary element of \({\rm D}(\varphi) = \mathbb{V}_p(\Omega)\).
Take \(V \in D_\varphi\) in \eqref{6.39} such as
\[
V(t) =
\left\{
\begin{aligned}
&V_0&&t\in I_{\overline{h}}:=[t_0-{\overline{h}}/2,t_0+{\overline{h}}/2),\\
&U(t)&&t \in [0,T]\setminus I_{\overline{h}}.
\end{aligned}
\right.
\]
Then dividing \eqref{6.39} by \({\overline{h}}>0\) and letting \({\overline{h}}\to0\), we get
\begin{equation}
(h+U,V_0-U)_{\mathbb{L}^2(\Omega)}
\leq \left(\varphi(V_0) + \frac{1}{2}|V_0|_{\mathbb{L}^2(\Omega)}^2\right)
-\left(\varphi(U)+\frac{1}{2}|U|_{\mathbb{L}^2(\Omega)}^2\right)
\end{equation}
holds for a.e. \(t \in[0,T]\), which implies that
\[
h+U=\partial\left(\varphi(U)+\frac{1}{2}|U|_{\mathbb{L}^2}^2\right).
\]
Hence we conclude
\begin{equation}\label{6.42}
h=\partial\varphi(U)\quad\mbox{a.e.}\ t \in [0,T].
\end{equation}
Thus, in view of \eqref{6.10}, \eqref{6.15} and \eqref{6.42}, we find that \(U\) satisfies
\[
\frac{dU}{dt}(t) +(\lambda + \alpha I)\partial\varphi(U)
+(\kappa+\beta I)\partial\psi(U) -\gamma U = F.
\]

As for the initial condition
\[
U(t) \to U_0\quad\mbox{in}\ \mathbb{L}^2(\Omega)\ \mbox{as}\ t \downarrow 0
\]
and the fact that $U \in {\rm C}([0,T];  \mathbb{L}^2(\Omega))$ can be verified by the arguments similar to that in the last part of the proof of Theorem \ref{main_result_2}.

\section{Proof of Theorem \ref{main_result_2}}

Let \(\{U^k_0\}_{k \in \mathbb{N}} \subset \mathbb{V}_p\) such that \(U^k_0 \to U_0\) in \(\mathbb{L}^2(\Omega)\) and \(\mathop{{\rm supp}}U^k_0 \subset \Omega_k\), where \(\Omega_k \subset \Omega\) satisfies (i) and (ii).
Let \(U^k = U\) be solutions of (ACGL)\(_p\) with \(\Omega = \Omega_k\) corresponding to initial data \(U^k_0\) given by Proposition \ref{main_result_bdd}.
Here we can assume without loss of generality that for all \(k \in \mathbb{N}\)
\begin{equation}
\label{bddUk01}
|U^k_0|_{\mathbb{L}^2} \leq |U_0|_{\mathbb{L}^2} + 1,.
\end{equation}
Then using the above boundedness, we can deduce the following a priori estimates by much the same arguments as before.
\begin{Lem}
\label{1st_energy_ACGL_k1}
Let $U$ be a solution of (ACGL)\(_p\) with \(\Omega = \Omega_k\) and initial data \(U^k_0\).
Then there exists a positive constant $C_1$ depending only on $\gamma$, $T$, $|U_0|_{\mathbb{L}^2}$ and $\int_0^T|F|_{\mathbb{L}^2}^2 dt$ satisfying
\begin{align}\label{1st_energy_ACGL_k1_1}
\sup_{t \in [0,T]}|U(t)|_{\mathbb{L}^2}^2
+\int_0^T \varphi(U(s)) ds
+\int_0^T \psi(U(s)) ds
\leq C_1.
\end{align}
\end{Lem}

\begin{Lem}[(cf. \cite{KOS1} Lemma 7.2)]
\label{2nd_energy_ACGL_k1}
Let $U$ be a solution of (ACGL)\(_p\) with \(\Omega = \Omega_k\) and initial data \(U^k_0\), and let $\left(\frac{\alpha}{\lambda}, \frac{\beta}{\kappa}\right) \in {\rm CGL}(c_q^{-1})$.
Then there exists a positive constant $C_2$ depending only on $\lambda, \kappa, \alpha, \beta, \gamma$, $T,\varphi(U_0), \psi(U_0)$, $|U_0|_{\mathbb{L}^2}$ and $\int_0^T|F|_{\mathbb{L}^2}^2 dt$ satisfying
\begin{align}
\notag
&\sup_{t \in [0,T]}t\varphi(U(t))
+ \sup_{t \in [0,T]}t\psi(U(t))\\
\label{2nd_energy_ACGL_k1_1}
&+ \int_0^T t\left| \frac{dU(t)}{dt}\right|_{\mathbb{L}^2}^2 dt
+ \int_0^T t|\partial \varphi(U(t))|_{\mathbb{L}^2}^2 dt
+ \int_0^T t|\partial \psi(U(t))|_{\mathbb{L}^2}^ 2dt
\leq C_2.
\end{align}
\end{Lem}

Lemma \ref{1st_energy_ACGL_k1} can be proved much the same way as in the proof of Lemmas \ref{1st_energy_AE_nu} and \ref{1st_energy_ACGL_k}.

To obtain \eqref{2nd_energy_ACGL_k1_1}, it suffices to we multiply \eqref{adfhgtyyjhds} and \eqref{sfdsldflsksfdhsl} by \(t \in (0, T)\) and integrate on \((0, T)\) with respect to \(t\).

By Lemmas \ref{1st_energy_ACGL_k1} and \ref{2nd_energy_ACGL_k1}, we can derive the following convergences of a subsequence \(\{U^{k_n}\} \subset \{U^k\}\) for any \(\delta \in (0, T)\):
\begin{alignat}{4}
\label{Uknw1}
U^{k_n} &\rightharpoonup U&&\quad\mbox{weakly in}\ {\rm L}^2(0, T;  \mathbb{L}^2(\Omega)),\\
\frac{dU^{k_n}}{dt} &\rightharpoonup \frac{dU}{dt}&&\quad\mbox{weakly in}\ {\rm L}^2(\delta, T;  \mathbb{L}^2(\Omega)),\\
\label{pphiUknw1}
\partial\varphi(U^{k_n}) &\rightharpoonup h&&\quad\mbox{weakly in}\ {\rm L}^2(\delta, T;  \mathbb{L}^2(\Omega)),\\
\partial\psi(U^{k_n}) &\rightharpoonup g&&\quad\mbox{weakly in}\ {\rm L}^2(\delta, T;  \mathbb{L}^2(\Omega)),
\end{alignat}
for some \(h, g \in {\rm L}^2(\delta, T; \mathbb{L}^2(\Omega))\).
Here we used the demiclosedness of \(\frac{d}{dt}\).

We repeat the same argument as above to obtain \(g = \partial\psi(U)\) fo a.e. \(t \in (0, T)\).

Multiplying \eqref{subdiffphi} by \(e^{-2t(\gamma + \lambda)}\) and integrating on \((\delta, T)\) we obtain
\[
\begin{aligned}
&\int_\delta^Te^{-2t(\gamma + \lambda)}(\left[\partial\varphi(U^{k'_n}(t))\right]^\sim + \tilde{U}^{k'_n}, V)_{\mathbb{L}^2}dt\\
&- \int_\delta^Te^{-2t(\gamma + \lambda)}(\left[\partial\varphi(U^{k'_n}(t))\right]^\sim + \tilde{U}^{k'_n}, \tilde{U}^{k'_n}(t))_{\mathbb{L}^2}dt\\
&\leq 
\begin{aligned}[t]
&\int_\delta^Te^{-2t(\gamma + \lambda)}\left\{\varphi(V)
+ \frac{1}{2}|V|_{\mathbb{L}^2}^2\right\}dt\\
& - \int_\delta^Te^{-2t(\gamma + \lambda)}\left\{\varphi(U^{k'_n}(t))
+\frac{1}{2}|\tilde{U}^{k'_n}|_{\mathbb{L}^2}^2\right\}dt.
\end{aligned}
\end{aligned}
\]

Again repeating the same argument, we obtain \eqref{IfU0} for a.e. \(t \in (0, T)\) so that it holds
\begin{align}
\notag
&\int_\delta^Te^{-2t(\gamma + \lambda)}(h + U, V - U)_{\mathbb{L}^2}dt\\
\label{subdiffint1}
&\leq
\begin{aligned}[t]
&\int_\delta^Te^{-2t\lambda(\gamma+\lambda)}\left\{\varphi(V)+\frac{1}{2}|V|_{\mathbb{L}^2}^2\right\}dt\\
&-\int_\delta^Te^{-2t\lambda(\gamma+\lambda)}\left\{\varphi(U(t))+\frac{1}{2}|U|_{\mathbb{L}^2}^2\right\}dt.
\end{aligned}
\end{align}
Taking same \(V \in D_\varphi\) as before,  we conclude \(h = \partial\varphi(U)\) a.e. \(t \in (0, T)\).

Then in order to complete the proof, 
   it suffices to check
   \begin{equation}\label{strogconv:U:L2}
    U(t) \rightarrow U_0
      \hspace{4mm}
          {\rm in}\ \mathbb{L}^2(\Omega)
             \hspace{2mm}
               {\rm as}\ t \downarrow 0.
   \end{equation}

 First we show
   $U(t) \rightharpoonup U_0$ weakly in
     $\mathbb{L}^2(\Omega)$.
 Multiplying (ACGL)\(_p\) with  initial data \(U^k_0\) by
   $W \in \mathbb{C}^\infty_0(\Omega)$, we have
\begin{equation}
\label{afytncvbfgdsap}
\begin{aligned}[b]
&        \frac{d}{dt}(\tilde{U}^k(t), W)_{\mathbb{L}^2} =
            \gamma ~\! (\tilde{U}^k(t), W)_{\mathbb{L}^2}
                + (F(t), W)_{\mathbb{L}^2}
\\
      &\hspace{10mm}    - ((\lambda+\alpha I) ~\! \left[\partial \varphi(U^k(t))\right]^\sim,
              W)_{\mathbb{L}^2}
                  - ((\kappa+\beta I) ~\! \partial \psi(\tilde{U}^k(t)),
                    W)_{\mathbb{L}^2}.
   \end{aligned}
\end{equation}
 Integrating \eqref{afytncvbfgdsap} over $(0,t)$ 
   and taking the absolute value, we get 
   \begin{align*} |(\tilde{U}^k(t)-U^k_0, W)_{\mathbb{L}^2}|
     &  \leq |\gamma||W|_{\mathbb{L}^2}
          \int_0^t |\tilde{U}^k(s)|_{\mathbb{L}^2}ds
            + |W|_{\mathbb{L}^2}\int_0^t |F(s)|_{\mathbb{L}^2}ds\\
              & \hspace{10mm}
                + (\lambda+|\alpha|)|\nabla W|_{\mathbb{L}^p}
                    \int_0^t |\nabla \tilde{U}^k(s)|_{\mathbb{L}^p}^{p-1}ds 
\\
    &  \hspace{20mm}
         + (\kappa+|\beta|)\int_0^t
            \int_\Omega |\tilde{U}^k(s)|_{\mathbb{R}^2}^{ q-1} |W|_{\mathbb{R}^2} dx ds.
   \end{align*}
 Then using H\"older's inequality
    and  Lemma \ref{1st_energy_ACGL_k1}, we obtain 
   \begin{align}
    \notag 
    & |(\tilde{U}^k(t)-U^k_0,  W)_{\mathbb{L}^2}|
         \leq |\gamma|\sqrt{C_1} ~\! |W|_{\mathbb{L}^2} ~\! t
            + |F(s)|_{ {\rm L}^2(0,t;  \mathbb{L}^2(\Omega))}
                 ~\! |W|_{\mathbb{L}^2} ~\! t^\frac{1}{2}
\\[2mm]
       &\quad   + (\lambda+|\alpha|) ~\! (p ~\! C_1)^\frac{p-1}{p} ~\! |\nabla W|_{\mathbb{L}^p} ~\! t^\frac{1}{p}
              + (\kappa+|\beta|) ~\! (q ~\! C_1)^\frac{q-1}{q} ~\! |W|_{\mathbb{L}^q} ~\! t^\frac{1}{q}.
   \end{align}
 Letting \(k = k_n'\) with $n \rightarrow \infty$,
   we obtain $|(U(t)-U_0, W)_{\mathbb{L}^2}|
     \leq C \min\left\{t^{\frac{1}{p}},t^\frac{1}{q}\right\}$ for sufficiently small $t>0$,
   which implies that $U(t) \rightarrow U_0$ in $\mathcal{D}'(\Omega)$.
 Since $\mathbb{C}^\infty(\Omega)
   \subset \mathbb{L}^2(\Omega)$ is dense, we find that 
     $U(t) \rightharpoonup U_0$ weakly in
       $\mathbb{L}^2(\Omega)$.

 Then, in order to derive \eqref{strogconv:U:L2}, it suffices to show that 
  $|U(t)|_{\mathbb{L}^2}^2
   \rightarrow |U_0|_{\mathbb{L}^2}^2$.
By the same argument as for \eqref{5.2} with the aid of \eqref{orth:IU}, we have for \(k \geq l\)
 \vspace{-1mm}
   \begin{equation*} 
   \label{groneq1}
     |U^k(t)|_{\mathbb{L}^2(\Omega_l)}^2
        \leq e^{(2\gamma_+ +1)t}\left\{ |U_0^k|_{
          \mathbb{L}^2}^2
             + \int_0^t |F(s)|_{\mathbb{L}^2}^2 ds
                \right\}\quad\forall t \in [0, T].
   \end{equation*}
Then by virtue of \eqref{dfhjdfhjdfh}, we let $k \rightarrow \infty$ to obtain
   \begin{equation} 
\begin{aligned}[t]
   \label{groneq2}
     |U(t)|_{\Omega_l}|_{\mathbb{L}^2(\Omega_l)}^2
&=     |[U(t)|_{\Omega_l}]^\sim|_{\mathbb{L}^2}^2\\
       & \leq e^{(2\gamma_+ +1)t}\left\{ |U_0|_{
          \mathbb{L}^2}^2
             + \int_0^t |F(s)|_{\mathbb{L}^2}^2 ds
                \right\}\quad\forall t \in [0, T].
\end{aligned}
   \end{equation}
It is clear that \(\{|[U(t,x)|_{\Omega_l}]^\sim|\}_{l \in \mathbb{N}}\) forms a pointwise monotonically increasing sequence.
Hence \eqref{groneq2} and  Beppo Levi's theorem yields that \([U(t,x)|_{\Omega_l}]^\sim\) converges to \(U(t,x)\) in \(\mathbb{L}^2(\Omega)\) for all \(t \in [0, T]\) and that \(U\) satisfies
   \begin{equation*} 
     |U(t)|_{\mathbb{L}^2}^2
        \leq e^{(2\gamma_+ +1)t}\left\{ |U_0|_{
          \mathbb{L}^2}^2
             + \int_0^t |F(s)|_{\mathbb{L}^2}^2 ds
                \right\}\quad\forall t \in [0, T].
   \end{equation*}
Here letting $t \downarrow 0$, we have
   $\overline{\lim}_{t \downarrow 0} |U(t)|_{
     \mathbb{L}^2}^2 \leq |U_0|_{\mathbb{L}^2}^2$.
 On the other hand, by virtue of the lower semicontinuity 
    of the norm with respect to the weak 
      convergence $U(t) \rightharpoonup U_0$,  
   we get $|U_0|_{\mathbb{L}^2}^2
             \leq \underline{\lim}_{t \downarrow 0} |U(t)|_{
               \mathbb{L}^2}^2$. Thus we can conclude that 
   $|U(t)|_{\mathbb{L}^2}^2 
     \rightarrow |U_0|_{\mathbb{L}^2}^2$.

\appendix
\section{Amalgam Spaces}\label{AmalSP}
In this section we investigate amalgam spaces defined in \S\ref{sec-2}.
Here we recall the notation of \(\mathbb{X}_p(\Omega)\):
\[
\mathbb{X}_p(\Omega) := \left\{u \in \mathbb{L}^2(\Omega);  \nabla u \in \left(\mathbb{L}^p(\Omega)\right)^N\right\}
\]
with norm
\[
|u|_{\mathbb{X}_p} :=
\left\{
\begin{aligned}
&\left[|u|_{\mathbb{L}^2}^p + |\nabla u|_{\mathbb{L}^p}^p\right]^{1/p}&&\mbox{for}\ p \geq 2,\\
&\left[|u|_{\mathbb{L}^2}^{p'} + |\nabla u|_{\mathbb{L}^p}^{p'}\right]^{1/{p'}}&&\mbox{for}\ \max\left\{1,\frac{2N}{N+2}\right\} < p \leq 2
\end{aligned}
\right.
\]
for all \(u \in \mathbb{X}_p(\Omega)\) and \((p,p')\) are H\"older conjugate exponents such that
\[
\frac{1}{p} + \frac{1}{p'} = 1.
\]

The aim of this appendix is to show that \(\mathbb{X}_p(\Omega)\) is uniformly convex.
For this, we prepare Clarkson's inequalities for vector valued functions.

\subsection{Clarkson's Inequalities for Vector Valued Functions}
We prove here the following two lemmas.
\begin{Lem}[Clarkson's first inequality]
Let \(2 \leq p < \infty\) and let \(N \in \mathbb{N}\).
Then
\begin{align}\label{aaaa}
\left\|\frac{f+g}{2}\right\|_{{\rm L}^p(\Omega)}^p \hspace{-2mm}+ \left\|\frac{f-g}{2}\right\|_{{\rm L}^p(\Omega)}^p
\leq
\left(\frac{1}{2}\|f\|_{{\rm L}^p(\Omega)}^{p'} + \frac{1}{2}\|g\|_{{\rm L}^p(\Omega)}^{p'}\right)^{\frac{1}{p-1}}
\quad\forall f, g \in ({\rm L}^p(\Omega))^N.
\end{align}
\end{Lem}

\begin{Lem}[Clarkson's second inequality]
Let \(1 \leq p < 2\), \(\frac{1}{p} + \frac{1}{p'} = 1\) and let \(N \in \mathbb{N}\).
Then
\begin{align}\label{bbbb}
\left\|\frac{f+g}{2}\right\|_{{\rm L}^p(\Omega)}^{p'} \hspace{-2mm}+ \left\|\frac{f-g}{2}\right\|_{{\rm L}^p(\Omega)}^{p'}
\leq
\left(
\frac{1}{2}\|f\|_{{\rm L}^p(\Omega)}^p + \frac{1}{2}\|g\|_{{\rm L}^p(\Omega)}^p
\right)^{\frac{1}{p-1}}
\ \forall f, g \in ({\rm L}^p(\Omega))^N.
\end{align}
\end{Lem}

\begin{proof}
First we claim that the following local Clarkson inequality holds, namely
\begin{align}\label{LClar}
\left\{\left|\frac{a + b}{2}\right|^p + \left|\frac{a-b}{2}\right|^p\right\}^{1/p}
\leq
\left\{
\frac{1}{2}|a|^{p'} + \frac{1}{2}|b|^{p'}
\right\}^{1/p'}
\end{align}
for all \(a, b \in \mathbb{R}^N\) and \(p \geq 2\) with \(\frac{1}{p} + \frac{1}{p'} = 1\).

To show \eqref{LClar}, it suffices to get
\begin{align}
\label{a}
\left|\frac{a+b}{2}\right|^p + \left|\frac{a-b}{2}\right|^p
\leq
\left\{\frac{1}{2}|a|^{p'}+\frac{1}{2}|b|^{p'}\right\}^{p/p'}
\quad\forall a,b \in \mathbb{R}^N.
\end{align}

If \(ab=0\), then \eqref{a} is obvious.
Let \(a=2y\), \(b=2z\) and divide \eqref{a} by \(|y|\) assuming \(0 \leq |z| \leq |y|\) and let \(x := \frac{z}{|y|}\).
Then without loss of generality, we can rewrite \eqref{a} in the form
\begin{align}
\label{b}
\begin{aligned}
\left|r + x\right|^p + \left|r - x\right|^p
&=
(1 + |x|^2 + 2(r\cdot x))^{p/2} + (1 + |x|^2 - 2(r\cdot x))^{p/2}\\
&\leq
2(1+|x|^{p'})^{p/p'},
\end{aligned}
\end{align}
where \(r, x \in \mathbb{R}^N\) satisfies \(|r| = 1\) and \(0 \leq |x| \leq 1\) and \((r\cdot x)\) is the innerproduct in \(\mathbb{R}^N\).
In view of the fact that \(|(r\cdot x)| \leq |x|\), we introduce the following function \(f(\cdot)\):
\begin{align}
f(\theta) = (1 + |x|^2 + 2|x|\theta)^{p/2} + (1 + |x|^2 - 2|x|\theta)^{p/2}\quad 0 \leq \theta \leq 1.
\end{align}
Since
\begin{align}
\frac{d}{d\theta}f(\theta)
=
pu\left[
(1 + |x|^2 + 2|x|\theta)^{p/2-1} - (1 + |x|^2 - 2|x|\theta)^{p/2-1}
\right]
\geq 0,
\end{align}
we get
\begin{equation}\label{asdfg}
f(\theta) \leq f(1)\quad\mbox{for all}\ \theta \in [0,1].
\end{equation}
For the case \((r\cdot x) \geq 0\), apply \eqref{asdfg} with \(\theta = (r\cdot x)\) and for the case \((r\cdot x) \leq 0\), apply \eqref{asdfg} with \(\theta = -(r\cdot x)\).
Then we have
\begin{align}
\label{c}
\begin{aligned}
&(1 + |x|^2 + 2(r\cdot x))^{p/2} + (1 + |x|^2 - 2(r\cdot x))^{p/2}\\
&\leq
(1 + |x|^2 + 2|x|)^{p/2} + (1 + |x|^2 - 2|x|)^{p/2}\\
&=
(1 + |x|)^p + (1 - |x|)^p.
\end{aligned}
\end{align}
Combining \eqref{b} with \eqref{c}, we find that to verify \eqref{a}, it suffices to show
\begin{align}
(1 + |x|)^p + (1 - |x|)^p \leq 2(1+|x|^{p'})^{p/p'},
\end{align}
which is nothing but the inequality (1)\(''\) in \cite{F1} for 1 dimension.

The lest part is much the same as in \cite{F1} due to Minkowski's integral inequality but we give the proof for the sake of completeness.
We recall extended Minkowski's integral inequality for measurable function \(f:S_1\times S_2 \to \mathbb{R}\) on two \(\sigma\)-finite measure space \((S_1, \mu), (S_2, \nu)\):
\begin{equation}\label{exMin}
\left\{
\int_{S_2}\left|
\int_{S_1}
|f(x,y)|^q\mu(dx)
\right|^{\frac{p}{q}}\nu(dy)
\right\}^{\frac{1}{p}}
\leq
\left\{
\int_{S_1}\left|
\int_{S_2}
|f(x,y)|^p\nu(dx)
\right|^{\frac{q}{p}}\mu(dy)
\right\}^{\frac{1}{q}}
,
\end{equation}
where \(0 < q \leq p\).

If one takes \(S_1 = \{1,2\}\) and \(\mu\) to be counting measure, one obtains
\begin{equation}\label{mnb}
\left\{
\int_{S_2}
\left\{
|f_1(y)|^q+|f_2(y)|^q
\right\}^{\frac{p}{q}}
\nu(dy)
\right\}^{\frac{1}{p}}
\leq
\left\{
\|f_1\|_{{\rm L}^p(S_2)}^q
+
\|f_2\|_{{\rm L}^p(S_2)}^q
\right\}^{\frac{1}{q}}\quad(q\leq p).
\end{equation}
We take \({\rm L}^p(\Omega)\)-norm of \eqref{LClar} with \(a = f(x), b = g(x)\) to obtain
\begin{equation}\label{lkj}
\left\{\left\|\frac{f+g}{2}\right\|_{{\rm L}^p(\Omega)}^p
+
\left\|\frac{f-g}{2}\right\|_{{\rm L}^p(\Omega)}^p
\right\}^{\frac{1}{p}}
\leq
\left\{
\int_\Omega
\left\{
\frac{1}{2}|f(x)|^{p'}+\frac{1}{2}|g(x)|^{p'}
\right\}^{\frac{p}{p'}}dx
\right\}^{\frac{1}{p}}.
\end{equation}
Since \(p \geq 2 \Leftrightarrow p' \leq p\), we combine \eqref{lkj} with \eqref{mnb} (\(q = p'\), \(S_2 = \Omega\), \(f_1 = \frac{1}{2^{1/p'}}f, f_2 = \frac{1}{2^{1/p'}}g\)) to obtain \eqref{aaaa}.

If one takes \(S_2 = \{1,2\}\) and \(\nu\) to be counting measure, one obtains
\begin{equation}\label{mna}
\left\{
\|f_1\|_{{\rm L}^q(S_1)}^p
+
\|f_2\|_{{\rm L}^q(S_1)}^p
\right\}^{\frac{1}{p}}
\leq
\left\{
\int_{S_1}
\left\{
|f_1(y)|^p+|f_2(y)|^p
\right\}^{\frac{q}{p}}
\nu(dy)
\right\}^{\frac{1}{q}}\quad(q\leq p).
\end{equation}
We in turn take \({\rm L}^{p'}(\Omega)\)-norm of \eqref{LClar} with \(a = f(x), b = g(x)\) to obtain
\begin{equation}\label{poi}
\begin{aligned}
&\left\{\int_\Omega
\left\{
\left|\frac{f(x)+g(x)}{2}\right|^p
+
\left|\frac{f(x)-g(x)}{2}\right|^p
\right\}^{\frac{p'}{p}}
dx
\right\}^{\frac{1}{p'}}\\
&\leq
\left(
\frac{1}{2}\|f\|_{{\rm L}^{p'}(\Omega)}^{p'}
+
\frac{1}{2}\|g\|_{{\rm L}^{p'}(\Omega)}^{p'}
\right)^{\frac{1}{p'}}.
\end{aligned}
\end{equation}
Since \(p \geq 2 \Leftrightarrow p' \leq p\), we combine \eqref{poi} with \eqref{mna} (\(q = p'\), \(S_1 = \Omega\), \(f_1 = \frac{f+g}{2}, f_2 = \frac{f-g}{2}\)) to obtain 
\begin{equation}\label{bbbc}
\left(
\left\|\frac{f+g}{2}\right\|_{{\rm L}^{p'}(\Omega)}^p
+
\left\|\frac{f-g}{2}\right\|_{{\rm L}^{p'}(\Omega)}^p
\right)^{\frac{1}{p}}
\leq
\left(
\frac{1}{2}\|f\|_{{\rm L}^{p'}(\Omega)}^{p'}
+
\frac{1}{2}\|g\|_{{\rm L}^{p'}(\Omega)}^{p'}
\right)^{\frac{1}{p'}}.
\end{equation}
For the case where \(1<p\leq 2\), since \(p' \geq 2\) we can apply \eqref{bbbc} by interchanging the roles of \(p\) and \(p'\) to obtain \eqref{bbbb}.
\end{proof}

\subsection{Uniform Convexity of \(\mathbb{X}_p(\Omega)\)}
We show that the Banach space \(\mathbb{X}_p(\Omega)\) is uniformly convex.

Before proving this, we prepare two inequalities for \(r \geq 1\):
\begin{align}
\label{convineq1}
(a-b)^r &\leq a^r - b^r&&\mbox{for}\ a\geq b\geq0,\\
\label{convineq2}
(a+b)^r &\leq 2^{r-1}\left(a^r + b^r\right)&&\mbox{for}\ a, b \geq 0.
\end{align}
In fact, if \(a=0\), then \eqref{convineq1} and \eqref{convineq2} are obvious.
For the case where \(a > 0\), dividing both sides of \eqref{convineq1} and \eqref{convineq2} by \(a > 0\), we find that it suffices to show the following:
\begin{align}
\label{convineq12}
(1-x)^r &\leq 1 - x^r&&\mbox{for}\ 0 \leq x \leq 1,\\
\label{convineq22}
(1+x)^r &\leq 2^{r-1}\left(1 + x^r\right)&&\mbox{for}\ x \geq 0.
\end{align}
Put
\[
g_\pm(x) = \frac{(1\pm x)^r}{1\pm x^r},
\]
then the differentiation of \(g\) gives
\[
g_\pm'(x) = \frac{r(1\pm x)^{r-1}}{(1\pm x^r)^2}
\left[\pm(1- x^{r-1})\right].
\]
Hence \(g_-(x)\) decreases monotonically for \(0 \leq x < 1\), which together with \(g_-(0) = 1\) means \eqref{convineq12}.
On the other hand \(g_+(x)\) takes its maximum at \(1\) with value \(2^{r-1}\), that is \eqref{convineq22}.

First we treat the case where \(p \geq 2\).
By the Clarkson's first inequality \eqref{aaaa}, we have
\begin{align}
\label{Cla1}
\left|\frac{u+v}{2}\right|_{\mathbb{L}^2}^2 + \left|\frac{u-v}{2}\right|_{\mathbb{L}^2}^2 &\leq \frac{1}{2}(|u|_{\mathbb{L}^2}^2 + |v|_{\mathbb{L}^2}^2),\\
\label{Cla2}
\left|\frac{\nabla u+\nabla v}{2}\right|_{\mathbb{L}^p}^p + \left|\frac{\nabla u-\nabla v}{2}\right|_{\mathbb{L}^p}^p &\leq \frac{1}{2}(|\nabla u|_{\mathbb{L}^p}^p + |\nabla v|_{\mathbb{L}^p}^p).
\end{align}
Inequalities \eqref{Cla1}, \eqref{convineq1} and \eqref{convineq2} lead
\begin{equation}
\label{Cla3}
\begin{aligned}
\left|\frac{u+v}{2}\right|_{\mathbb{L}^2}^p 
&\leq \left\{\frac{1}{2}(|u|_{\mathbb{L}^2}^2 + |v|_{\mathbb{L}^2}^2) - \left|\frac{u-v}{2}\right|_{\mathbb{L}^2}^2\right\}^{\frac{p}{2}}\\
&\leq \frac{1}{2}(|u|_{\mathbb{L}^2}^p + |v|_{\mathbb{L}^2}^p) - \left|\frac{u-v}{2}\right|_{\mathbb{L}^2}^p.
\end{aligned}
\end{equation}
We combine \eqref{Cla2} with \eqref{Cla3} to obtain
\[
\left|\frac{u+v}{2}\right|_{\mathbb{L}^2}^p + \left|\frac{\nabla u+\nabla v}{2}\right|_{\mathbb{L}^p}^p
\leq
\begin{aligned}[t]
&\frac{1}{2}(|u|_{\mathbb{L}^2}^p + |\nabla u|_{\mathbb{L}^p}^p + |v|_{\mathbb{L}^2}^p + |\nabla v|_{\mathbb{L}^p}^p)\\
&- \left|\frac{u-v}{2}\right|_{\mathbb{L}^2}^p - \left|\frac{\nabla u-\nabla v}{2}\right|_{\mathbb{L}^p}^p,
\end{aligned}
\]
whence follows the uniform convexity of \(\mathbb{X}_p(\Omega)\) for \(p \geq 2\).

As for the case where \(1 < p < 2\), instead of \eqref{Cla2}, we can derive by \eqref{bbbb} the following Clarkson's second inequality:
\begin{equation}
\label{Cla4}
\left|\frac{\nabla u+\nabla v}{2}\right|_{\mathbb{L}^p}^{p'} + \left|\frac{\nabla u-\nabla v}{2}\right|_{\mathbb{L}^p}^{p'} \leq \left(\frac{1}{2}(|\nabla u|_{\mathbb{L}^p}^p + |\nabla v|_{\mathbb{L}^p}^p)\right)^{\frac{1}{p-1}}.
\end{equation}
Combining \eqref{Cla4} with \eqref{convineq2} with \(r=p'\), we obtain
\begin{equation}
\label{Cla5}
\left|\frac{\nabla u+\nabla v}{2}\right|_{\mathbb{L}^p}^{p'} + \left|\frac{\nabla u-\nabla v}{2}\right|_{\mathbb{L}^p}^{p'}
\begin{aligned}[t]
&\leq \frac{1}{2}(|\nabla u|_{\mathbb{L}^p}^{\frac{p}{p-1}} + |\nabla v|_{\mathbb{L}^p}^{\frac{p}{p-1}})\\
&=\frac{1}{2}(|\nabla u|_{\mathbb{L}^p}^{p'} + |\nabla v|_{\mathbb{L}^p}^{p'}).
\end{aligned}
\end{equation}
Moreover, since \(p' > 2\), \eqref{Cla3} holds true with \(p\) replaced by \(p'\), i.e.,
\begin{equation}
\label{Cla6}
\left|\frac{u+v}{2}\right|_{\mathbb{L}^2}^{p'} 
\leq \frac{1}{2}(|u|_{\mathbb{L}^2}^{p'} + |v|_{\mathbb{L}^2}^{p'}) - \left|\frac{u-v}{2}\right|_{\mathbb{L}^2}^{p'}.
\end{equation}
Now combining \eqref{Cla5} with \eqref{Cla6}, we get
\[
\left|\frac{u+v}{2}\right|_{\mathbb{L}^2}^{p'} + \left|\frac{\nabla u+\nabla v}{2}\right|_{\mathbb{L}^p}^{p'}
\leq
\begin{aligned}[t]
&\frac{1}{2}(|u|_{\mathbb{L}^2}^{p'} + |\nabla u|_{\mathbb{L}^p}^{p'} + |v|_{\mathbb{L}^2}^{p'} + |\nabla v|_{\mathbb{L}^p}^{p'})\\
&- \left|\frac{u-v}{2}\right|_{\mathbb{L}^2}^{p'} - \left|\frac{\nabla u-\nabla v}{2}\right|_{\mathbb{L}^p}^{p'},
\end{aligned}
\]
which means \(\mathbb{X}_p(\Omega)\) is uniformly convex also for \(1<p < 2\).


\end{document}